\documentclass[smallextended]{svjour3}

\usepackage[T1]{fontenc}
\usepackage[utf8]{inputenc}
\usepackage[english]{babel}
\usepackage{latexsym,amsfonts,amssymb,amsmath,amscd,float, geometry, hyperref, tikz}
\usepackage{color}

\newcommand\Ba{\mathfrak{B}}
\newcommand\A{\mathcal A}
\newcommand\B{\mathcal B}
\newcommand\Z{\mathbb Z}
\newcommand\Q{\mathbb Q}
\newcommand\R{\mathbb R}
\newcommand\N{\mathbb N}

\newcommand\F{F_\ast}
\newcommand\K{\mathcal{K}}
\newcommand\V{\mathcal{V}}
\newcommand\s{\sigma}
\newcommand{\Ms}{\mathcal M_\sigma}
\newcommand\azd{\mathcal A^{\mathbb Z^d}}
\newcommand\bzd{\mathcal B^{\mathbb Z^d}}
\newcommand\dm{d_\mathcal M}
\newcommand\Ball{\mathbf{B}}
\newcommand\meas[1]{\widehat{\delta_{#1}}}
\newcommand\dinf{d_{\infty}}

\newcommand\Supp{\mathbf{Supp}}
\newcommand\Int{\mathbf{Int}}
\newcommand\Ext{\mathbf{Ext}}
\newcommand\Col{\mathbf{Col}}

\newcommand{\sti}{\vbox to 7pt{\hbox{\includegraphics{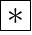}}}\hspace{.1em}}

\newcommand{\std}{\vbox to 7pt{\hbox{\includegraphics{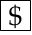}}}\hspace{.1em}}
\newcommand{\stdp}{\vbox to 7pt{\hbox{\includegraphics{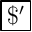}}}\hspace{.1em}}
\newcommand{\stdsub}{\vbox to 7pt{\hbox{\includegraphics[width=0.25cm]{figures/state_std}}}\hspace{.1em}}

\newcommand{\stp}{\vbox to 7pt{\hbox{\includegraphics{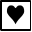}}}\hspace{.1em}}

\newcommand{\memb}{\vbox to 7pt{\hbox{\includegraphics{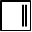}}}\hspace{.1em}}
\newcommand{\membco}{\vbox to 7pt{\hbox{\includegraphics{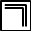}}}\hspace{.1em}}

\newcommand{\death}{\vbox to 7pt{\hbox{\includegraphics{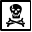}}}\hspace{.1em}}
\newcommand{\stg}{\vbox to 7pt{\hbox{\includegraphics{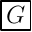}}}\hspace{.1em}}
\newcommand{\Bl}{\#}
\newcommand{\cun}{\vbox to 8pt{\hbox{\includegraphics[scale=1.1]{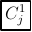}}}\hspace{.1em}}
\newcommand{\cunde}{\vbox to 8pt{\hbox{\includegraphics[scale=1.1]{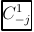}}}\hspace{.1em}}
\newcommand{\cundemi}{\vbox to 8pt{\hbox{\includegraphics[scale=1.1]{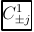}}}\hspace{.1em}}
\newcommand{\cde}{\vbox to 8pt{\hbox{\includegraphics[scale=1.1]{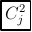}}}\hspace{.1em}}
\newcommand{\cdede}{\vbox to 8pt{\hbox{\includegraphics[scale=1.1]{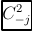}}}\hspace{.1em}}

\newcommand\Abirth{\A_{\mathrm{birth}}}
\newcommand\Agrowth{\A_{\mathrm{growth}}}
\newcommand\Aorga{\A_{\mathrm{orga}}}
\newcommand\Aevol{\A_{\mathrm{evol}}}
\newcommand\Acomp{\A_{\mathrm{comp}}}
\newcommand\Acopy{\A_{\mathrm{copy}}}
\newcommand\pbirth{p_{\mathrm{birth}}}
\newcommand\pgrowth{p_{\mathrm{growth}}}
\newcommand\porga{p_{\mathrm{orga}}}
\newcommand\pevol{p_{\mathrm{evol}}}
\newcommand\pcomp{p_{\mathrm{comp}}}
\newcommand\pcopy{p_{\mathrm{copy}}}
\newcommand\pmain{p_{\mathrm{main}}}

\DeclareMathOperator{\supp}{supp}

\spdefaulttheorem{fact}{Fact}{\bf}{\it}
\smartqed

\journalname{TOCS}

\title{Characterisation of limit measures\\ of higher-dimensional cellular automata\thanks{Research partially supported by the FONDECYT Postdoctorado Proyecto 3130496 and by Basal project No. PFB-03 CMM, Universidad de Chile}}
\author{Martin Delacourt \and Benjamin Hellouin de Menibus}
\date{}
\institute{Martin Delacourt (Corresponding author)\at CMM, Universidad de Chile, CNRS, Av. Blanco Encalada 2120, Santiago, Chile
  \and
  Benjamin Hellouin de Menibus\at CMM, Universidad de Chile, CNRS, Av. Blanco Encalada 2120, Santiago, Chile\\ Departamento de Matemáticas, Universidad Andres Bello, Chile}

\begin{document}
\maketitle

\begin{abstract}
We consider the typical asymptotic behaviour of cellular automata of higher dimension ($\geq 2$). That is, we take an initial configuration at random according to a Bernoulli (i.i.d) probability measure, iterate some cellular automaton, and consider the (set of) limit probability measure(s) as $t\to\infty$. In this paper, we prove that limit measures that can be reached by higher-dimensional cellular automata are completely characterised by computability conditions, as in the one-dimensional case. This implies that cellular automata have the same variety and complexity of typical asymptotic behaviours as Turing machines, and that any nontrivial property in this regard is undecidable (Rice-type theorem). These results extend to connected sets of limit measures and Cesàro mean convergence. The main tool is the implementation of arbitrary computation in the time evolution of a cellular automata in such a way that it emerges and self-organises from a random configuration.
\end{abstract}

\bigskip

Cellular automata are discrete dynamical systems defined by a local rule, introduced in the 40s by John von Neumann \cite{Neumann}. They model a large variety of discrete systems and are linked with various areas of mathematics and computer science, in particular computation theory, complex systems, ergodic theory and combinatorics.

One of the main catalysts of the study of cellular automata was their surprisingly complex and organised behaviours, even when iterated on configurations with no particular structure (e.g. chosen at random). To formalise these observations, many authors tried to describe their asymptotic behaviour by considering the limit set, which is the set of configurations that can be reached after arbitrarily many steps. These sets were shown to have potentially high computational complexity \cite{Maass-1995,Ballier-Guillon-Kari-2011}, and any nontrivial property on them is undecidable \cite{Kari-1994}. These observations built a bridge between the variety of dynamical behaviours and the computational content of the model. Nevertheless, the problem of characterising which sets can be limit sets of CA remains open. \bigskip

 In 2000, K\r{u}rka and Maass argued that limit sets did not provide a good description of empirical observations and introduced instead a measure-theoretical version of these sets \cite{km-2000}. The idea of $\mu$-limit sets is to choose the initial configuration at random, according to some probability measure $\mu$, to iterate the cellular automaton on this configuration and to consider all patterns whose probability to appear does not tend to 0. In the one-dimensional case, this approach yielded similar results of high complexity and undecidability \cite{bpt-2006,Boyer-Delacourt-Sablik-2010,Delacourt-2011,bdpst}. Although these two families of results appear similar and both require sophisticated constructions inside cellular automata, they provide insight about different kinds of dynamics (topological vs. measure-theoretical) and computational power (deterministic vs. probabilistic).
 
 In \cite{HellouinSablik}, H. and Sablik extended this approach to consider the limit probability measure(s). Still in the one-dimensional case, they provided a computational characterisation of the limit measures reachable by cellular automata, generalising the previous results. 
\bigskip
 
This article is an extended version of \cite{Delacourt-Hellouin}. In \emph{op.cit}, we aimed at extending the previous results to the two-dimensional setting. More precisely, we characterised all subshifts that can be $\mu$-limit sets of CA when $\mu$ is the uniform Bernoulli measure. The proof works by an explicit construction inspired by the one-dimensional constructions of \cite{bdpst,HellouinSablik}, although the higher dimensional setting has many specific challenges. In the present article, this two-dimensional construction is generalised to a $d$-dimensional space for any $d>2$; furthermore, through a more careful analysis, we are able to characterise reachable limit measures, which is a more general result. As a corollary, we obtain an undecidability result on properties of limit measures, and cover as well Cesàro mean convergence and the case where the limit measure is not unique.
 
Section~\ref{sec:definitions} is devoted to general definitions. In Section~\ref{sec:computability}, we introduce more specific computability tools, and in particular computability restrictions on possible limit probability measures. In Section~\ref{sec:construction}, we describe, process by process, the main technical construction that is the core of the proof of our results. Section~\ref{sec:results} contains the main results that are the corollaries of this construction, mainly:

\begin{theorem}[Main result]
The measures $\nu\in\Ms(\azd)$ for which there exist:
\begin{itemize}
\item an alphabet $\B\supset \A$,
\item a cellular automaton $F:\bzd\to\bzd$, and
\item a non-degenerate Bernoulli measure $\mu\in\Ms(\bzd)$
\end{itemize}
such that $F^t\mu\xrightarrow[t\to\infty]{}\nu$, are exactly the limit-computable measures.
\end{theorem}

This result holds even when the initial measure is chosen to be uniform. As a corollary, we show that any nontrivial property on limit measures is undecidable. These result extend to connected sets of limit measures and convergence in Cesàro mean.
 
\section{Definitions}
\label{sec:definitions}
\subsection{Symbols, configurations and cellular automata}

Let $\A$ be a finite set of symbols called \emph{alphabet}. For $d>0$, let $\azd$ be the space of $d$-dimensional \emph{configurations}. 

On $\Z^d$, we define the basis vectors $e_i = (\delta_i(k))_{0< k\leq d}$ (Kronecker deltas), that is, the vector worth $0$ on all coordinates except the $i$-th where it is worth $1$. Denote $\mathcal{U}nit(d)=\{\sum_{1\leq j\leq d}\delta_je_j\neq 0:\forall j, \delta_j\in\{-1,0,1\}\}$ and $\mathcal{H}yp(d)$ the set of hyperplanes that have a normal vector in $\mathcal{U}nit(d)$; these hyperplanes have a basis of $d-1$ vectors in $\mathcal{U}nit(d)$.

We will use the following distances between points of $\Z^d$: 
\[\forall x,y\in\Z^d,\hspace{1cm} d_\infty(x,y) = \max_{1\leq i\leq d} |x_i-y_i| \quad \text{and}\quad d_1(x,y) = \sum_{1\leq i\leq d} |x_i-y_i|.\]
An $\infty$-path is a sequence of points $z_1, \dots, z_k$ such that $d_\infty(z_i,z_{i+1})=1$ for any $i$. An $\infty$-connected set is a subset of $\Z^d$ such that any pair of points are connected by an $\infty$-path. Define similarly $1$-paths and $1$-connected subsets.

If we endow $\azd$ with the product topology of the discrete topology on $\A$, then $\azd$ is a Cantor set (compact, perfect and totally disconnected). This topology is also metrisable, for example using the \emph{Cantor metric}:
\[\forall x,y\in\azd,\ d_C(x,y) = 2^{-\delta_{x,y}}\quad\mbox{where}\quad\delta_{x,y} = \min\{||i||_\infty\ :\ x_i\neq y_i\}.\]
For a subset $U\subset \Z^d$, denote $x_U \in \A^U$ the restriction of $x$ to $U$. Denote $\A^\ast = \bigcup_{\underset{finite}{U\subset \Z^d}}\A^U$ the set of finite \emph{patterns}. For a pattern $w \in \A^U$, denote its \emph{support} $\supp(w) = U$, and its \emph{dimension} is the smallest $d$ such that $\supp(w)$ is isomorphic to a subset of $\Z^d$. We say a pattern is \emph{cubic}, respectively \emph{rectangular}, if its support is a $d-cube$, resp. a $d-box$ (Cartesian product of intervals).

The \emph{cylinder} defined by a pattern $u\in\A^\ast$ and a position $i\in\Z^d$ is $[u]_i = \{x\in\azd\ :\ x_{i+\supp(u)}=u\}$. For simplicity we sometimes write $[u]$ for $[u]_{(0,\dots,0)}$.

The \emph{shift map}, or \emph{shift}, is defined as:
\[\forall i\in\Z^d,\ \s_i(x) = (x_{i+j})_{j\in\Z^d}.\]

A \emph{subshift} is a closed subset of $\azd$ invariant under all shifts. Given a cubic pattern $u \in \A^{[0,n-1]^d}$, define the \emph{periodic configuration} $\mathstrut^\infty u^\infty$ by $^\infty u_{[0,n-1]^d}^\infty = u$ and $\s_{e_k}^n(\mathstrut^\infty u^\infty) = \mathstrut^\infty u^\infty$ for every $k\in [1,d]$.

A \emph{cellular automaton} (or CA) is a continuous function $F : \azd \to \azd$ that commutes with all shifts ($F\circ \s_{e_k} = \s_{e_k}\circ F$ for every $k$). By the Curtis-Hedlund-Lyndon theorem \cite{Hedlund}, it can be defined equivalently as a function $F(x) = (f((x_j)_{j\in (i+\mathcal N)}))_{i\in\Z^d}$ where $\mathcal N\subset \Z^d$ is a finite set called \emph{neighbourhood} and $f : \A^{\mathcal N} \to \A$ is called a \emph{local rule}.

\subsection{Probability measures on $\azd$}
Let $\Ba$ be the Borel $\sigma$-algebra of $\azd$ and $\mathcal{M}(\azd)$ the set of probability measures on $\azd$ defined on the $\sigma$-algebra $\Ba$. In this article, we focus on $\Ms(\azd)$ the set of \emph{$\s$-invariant probability measures} on $\azd$, 
that is to say, the measures $\mu$ such that $\mu(\s_k^{-1}(B))=\mu(B)$ for all $B\in\Ba$ and $k \in \Z^d$. Cylinders corresponding to finite patterns form a base of the topology. Since $\mu([u]_i) = \mu([u])$ for any $i\in\Z^d$ and $\mu\in\Ms(\azd)$, $\mu$ is entirely characterised by $\{\mu([u])\ :\ u\in\A^\ast\}$; actually, considering only cubic patterns is enough.

We endow $\Ms(\azd)$ with the \emph{weak$^\ast$} (or \emph{weak convergence}) topology:
\[\mu_n\xrightarrow[n\in\infty]{}\mu \quad\Longleftrightarrow\quad \forall u\in\A^\ast,\ \mu_n([u])\xrightarrow[n\to\infty]{} \mu([u]).\] In the weak$^{\ast}$ topology, $\Ms(\azd)$ is compact and metrisable. A metric is defined by
\[\dm(\mu,\nu)=\sum_{n\in\mathbb N}\frac{1}{2^n}\max_{u\in\A^{[0,n]^d}}|\mu([u])-\nu([u])|.\]

Define the \emph{ball} centred on $\mu\in\Ms(\azd)$ of radius $\varepsilon >0$ as
\[\Ball(\mu,\varepsilon)=\left\{\nu\in\Ms(\azd) : \dm(\mu,\nu)\leq\varepsilon \right\}.\]

Let us define some examples that we use throughout the article.

The \emph{Bernoulli measure} $\mu_{\lambda}$ associated to some vector $\lambda=(\lambda_a)\in[0;1]^\A$ satisfying $\sum_{a\in\A}\lambda_a=1$ is defined by \[\mu_{\lambda}([u_0\dots u_n])=\lambda_{u_0}\cdots\lambda_{u_n}\qquad \textrm{for all }u_0\dots u_n\in\A^\ast.\]

The \emph{Dirac measure} supported by $x\in\azd$ is defined as $\delta_x(A)= \mathbf 1_{x\in A}$. Generally $\delta_x$
is not $\s$-invariant. However, for any cubic pattern $w\in\A^{[0,n]^d}$, it is possible to define the \emph{$\s$-invariant measure supported by $^{\infty}w^{\infty}$} by taking the mean of the Dirac measures on the orbit under $\sigma$:
\[\meas{w}=\frac1{(n+1)^d}\sum_{i\in[0,n]^d}\delta_{\s_i(\mathstrut^\infty w^\infty)}.\]

The set of measures $\left\{\meas{w}:w\in\A^{\ast}\right\}$ is dense in $\Ms(\azd)$~\cite{Petersen-1983-livre}.

\subsubsection{Action of a cellular automaton on $\Ms(\azd)$ and limit measures} 
Let $F : \azd\to\azd$ be a cellular automaton and $\mu\in\Ms(\azd)$. Define the \emph{image measure} $\F\mu$ by
$\F\mu(A)=\mu(F^{-1}(A))$ for all $A\in\Ba$.
Since $F$ is $\s$-invariant, that is to say $F\circ\s=\s\circ F$, one deduces that $\F(\Ms(\azd)) \subset \Ms(\azd)$.
This defines a continuous application $\F:\Ms(\azd)\to\Ms(\azd)$.  

We consider in particular $\F^t\mu$ the iterated image of $\mu$ by $\F$. Since $\Ms(\azd)$ is compact in the weak$\ast$ topology, the sequence $(\F^t\mu)_{t\in\mathbb N}$ admits a set of limit points denoted $\V(F,\mu)$ and called the \emph{$\mu$-limit set of measures of $F$}. When $\V(F,\mu)$ is a singleton, i.e. when $\F^t\mu \underset{n\to\infty}\longrightarrow \nu$, we say $\nu$ is the \emph{limit measure} of $F$ starting on $\mu$.

This name stems from the standard $\mu$-limit set of $F$ defined as $\bigcup_{\nu\in\V(F,\mu)}\supp(\nu)$.


\section{Computability}
\label{sec:computability}
We now introduce the computability notions that are needed to state our main results. This exposition is very similar to the one found in \cite{HellouinSablik}, which was later expanded in \cite{These}, for the one-dimensional case. Indeed, most of the definitions and proofs only rely on the fact that the space is metric and separable, properties for which the increase in dimension is irrelevant. We omit those proofs that can be obtained by a straightforward substitution ($\A^\Z \to \azd$) from the proofs found in these references.

\subsection{Turing machines}
Turing machines are a standard and robust tool to define the computability of mathematical operations.
In the usual model, they have access to a one-dimensional, one- or two-sided infinite memory tape. In order to simplify some constructions, we consider in this article that the tape is $d$-dimensional and infinite in all directions. This does not affect the computing power of the model.\bigskip

A \emph{Turing machine} $\mathcal{TM}=(Q,\Gamma,\#,q_0,\delta,Q_F)$ is defined by:
\begin{itemize}
\item $\Gamma$ a finite alphabet, with a blank symbol $\#\notin\Gamma$. Initially, a $d$-dimensional infinite memory tape
is filled with $\#$, except for a finite region (the input), and a computing head is located at coordinate $(0, \dots, 0)$;
\item $Q$ a finite set of states, with an initial state $q_0\in Q$;
\item $\delta : (Q\cup\#)\times\Gamma\to (Q\cup\#)\times\Gamma\times\{\pm e_i\}_{1\leq i\leq d}$ the transition function. Given
the current state and the letter it reads on the tape --- which depends on its current position --- the function returns the new state, the letter to be written on the tape at current position, and the vector by which the head moves.
\item $Q_F\subset Q$ the set of final states --- when a final state is
reached, the computation stops and the output is the contents of the tape. 
\end{itemize}
A function $f : \A^\ast\to \A^\ast$ is \emph{computable} if there exists a Turing machine working on an alphabet
$\Gamma\supset\A$ that, on any input $w\in\A^\ast$, eventually stops and outputs $f(w)$.

\subsection{Computability of functions mapping countable sets}

To generalise this definition to functions mapping arbitrary countable sets $X \to Y$, we need to define an \emph{encoding}, that is, an alphabet $\A_X$ together with a bijection between $X$ and some subset of $\A_X^\ast$, and similarly for $Y$. Then the computability of a function $X\to Y$ is defined up to some encoding. However, in practice, reasonable encodings yield the same notion of computability. To simplify notations, we fix some canonical encodings for the rest of the paper :

\begin{description}
\item[$\Z$ (or $\mathbb N$):] Take $\A_\Z = \{0,1\}$ and encode an element $k\in\Z$ as its binary expansion surrounded by blank symbols.
\item[Product $X\times Y$:] Take $\A_{X\times Y} = \A_X\times\A_Y$ and encode $(x,y)$ as the product of encodings for $x$ and $y$.
\end{description}
Using this last case, we define a canonical encoding for $\Q$ as the canonical encoding for $\mathbb N\times \Z$, up to the bijection $\frac pq \mapsto (p,q)$ (with $p, q$ irreducible).

Furthermore, we define the computability of a set $K\subset X$ as the computability of the function $1_K : X\to \mathbb N$.

\subsection{Computability of probability measures}

As we mentioned above, a probability measure $\mu \in\Ms(\azd)$ is entirely described by the value of $\mu([u])$ for all $u\in\A^\ast$. In other words, an element of $\Ms(\azd)$ is described by a function $\A^\ast\to\R$. Since $\R$ is not countable, the standard ways to define notions of computability is to consider approximations by elements of $\Q$.\bigskip

A measure $\mu\in\Ms(\azd)$ is \emph{computable} if there exists a computable function $f:\A^{\ast}\times\mathbb N\to\Q$ such that
\[|\mu([u])-f(u,n)|<2^{-n} \qquad \textrm{for all }u\in\A^{\ast} \textrm{ and }n\in\mathbb N.\]  
It is \emph{limit-computable} if there exists a computable function $f:\A^{\ast}\times\mathbb N\to\Q$ such that
\[|\mu([u])-f(u,n)|\underset{n\to\infty}\longrightarrow 0 \qquad \textrm{for all }u\in\A^{\ast}.\]  

Additionally we define the notion of a \emph{uniformly computable sequence}. Informally, it means that a sequence of objects can be computed by a single algorithm which, given $n\in\mathbb N$ as input, returns a description of the $n$-th object of the sequence.

Formally, a sequence of measures $(\mu_i)_{i\in\N}$ is \emph{uniformly computable} iff there exists a computable function $f:\N\times \A^{\ast}\times \mathbb N\to\Q$ such that: \[|\mu_i([u])-f(n,u,i)|<2^{-n}\qquad\mbox{for all }u\in\A^\ast\mbox{ and } n,i\in\mathbb N^2.\]
It is easy to see that the limit of a uniformly computable sequence of measures is limit-computable (but not necessarily computable since the rate of convergence of $\mu_i$ to $\mu$ is not known). 

\begin{proposition}[Approximation by measures supported by periodic orbits]\label{prop:ApproximComputableMeasure}~

These notions can be defined in another equivalent way: 
\begin{itemize}
\item[(i)] A measure $\mu\in\Ms(\azd)$ is computable if and only if there exists a computable function $f:\mathbb N\to\A^{\ast}$  such that
$\dm\left(\mu,\meas{f(n)}\right)\leq 2^{-n}$
for all $n\in\mathbb N$.
\item[(ii)] A measure $\mu\in\Ms(\azd)$ is limit-computable if and only if there exists a computable function $f:\mathbb N\to\A^{\ast}$ such that
$\underset{n\to\infty}{\lim}\meas{f(n)}=\mu$.
\end{itemize}
\end{proposition}

Notice the parallel with the definition of the computability of a real: in both cases, an object is computable if it is approximated by a uniformly computable sequence of elements taken from a dense subset ($\Q$ and the measures supported by periodic orbits, respectively) with a known rate of convergence.

\subsection{Action of a cellular automaton on computable measures}

\begin{proposition}[First computability obstruction]\label{prop:OneStep}
Let $F:\azd\to\azd$ be a cellular automaton and $\mu\in\Ms(\azd)$ be a computable measure. Then $(\F^t\mu)_{t\in\N}$ is a
uniformly computable sequence of measures. In particular, if $\F^t\mu\underset{t\to\infty}{\longrightarrow}\nu$ then
$\nu$ is limit-computable.
\end{proposition}

In general, $\F^t\mu$ does not have a single limit point, but a compact set of accumulation points. To obtain a similar obstruction, we extend our computability definitions to those objects.

\subsection{Compact sets in computable analysis}

Extending naively the definition for countable sets using the characteristic function does not work since the set of inputs would not be countable. Instead, we use a general definition for metric spaces that possess a countable dense subset, $(\meas{w})_{w\in\A^\ast}$ in the case of $\Ms(\azd)$.

A closed set $\K\subset \Ms(\azd)$ is \emph{computable} if the set
 $\left\{(w,r)\in\A^\ast\times \Q\ :\ \overline{\Ball(\meas{w}, r)} \cap \K \neq \emptyset \right\}$
is computable (as a countable set), that is, if its characteristic function is.

However, the set of limit points of the sequence $(\F^t\mu)_{t\in\N}$, where $\mu$ is computable, is not necessarily computable (or even limit-computable). We extend our definitions to the first steps of the so-called arithmetical hierarchy, first on countable spaces, then on closed subsets of $\Ms(\azd)$.

Let $X,Y$ be two countable sets, with $Y$ being ordered.

A sequence of functions $(f_i:X\to Y)_{i\in\N}$ is \emph{uniformly computable} if $(i,x)\mapsto f_i(x)$ is computable.

A function $f:X\to Y$ is \emph{$\Pi_2$-computable} (resp. \emph{$\Sigma_2$-computable}) if $f = \inf_{i\in\N}\sup_{j\in\N}f_{i,j}$ (resp. $f = \sup_{i\in\N}\inf_{j\in\N}f_{i,j}$), where $(f_{i,j})_{(i,j)\in\N^2}$ is a uniformly computable sequence of functions.

 A closed set $\K\subset \Ms(\azd)$ is \emph{$\Pi_2$-computable} if the set
 \[\left\{(w,r)\in\A^\ast\times \Q\ :\ \overline{\Ball(\meas{w}, r)} \cap \K \neq \emptyset \right\}\]
is $\Pi_2$-computable, that is, its characteristic function is.

\begin{remark}
The symmetric notions of $\Pi_2$- and $\Sigma_2$-computability come from an analogy with the real arithmetical hierarchy~\cite{Zheng-Weihrauch-2001,Ziegler-2005}. These definitions extend naturally to $\Pi_n$- and $\Sigma_n$-computability. Other equivalent definitions exists, see for example \cite{HellouinSablik} for $\Pi_2$-computability or \cite{These} for a more general result.
\end{remark}

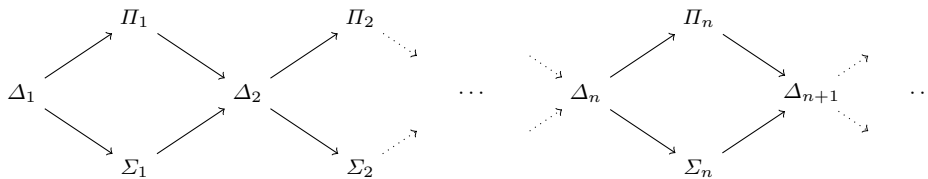
\begin{figure}[!ht]
\begin{center}
\begin{tikzpicture}
\draw node(delta1) at (0,0) {$\Delta_1$};
\draw node(sigma1) at (1.5,-1) {$\Sigma_1$};
\draw node(pi1) at (1.5,1) {$\Pi_1$};
\draw node(delta2) at (3,0) {$\Delta_2$};
\draw node(sigma2) at (4.5,-1) {$\Sigma_2$};
\draw node(pi2) at (4.5,1) {$\Pi_2$};
\draw node at (6,0) {$\cdots$};
\draw node(deltan) at (7.5,0) {$\Delta_n$};
\draw node(sigman) at (9,-1) {$\Sigma_n$};
\draw node(pin) at (9,1) {$\Pi_n$};
\draw node(deltanp1) at (10.5,0) {$\Delta_{n+1}$};
\draw node at (12,0) {$\cdots$};
\draw[->] (delta1) -- (sigma1);
\draw[->] (sigma1) -- (delta2);
\draw[->] (delta1) -- (pi1);
\draw[->] (pi1) -- (delta2);
\draw[->] (delta2) -- (pi2);
\draw[->] (delta2) -- (sigma2);
\draw[->, dotted] (sigma2) -- (5.25, -.5);
\draw[->, dotted] (pi2) -- (5.25, .5);
\draw[->, dotted] (6.75, .5)  -- (deltan);
\draw[->, dotted] (6.75, -.5)  -- (deltan);
\draw[->] (deltan) -- (sigman);
\draw[->] (deltan) -- (pin);
\draw[->] (sigman) -- (deltanp1);
\draw[->] (pin) -- (deltanp1);
\draw[->, dotted] (deltanp1) -- (11.25, -.5);
\draw[->, dotted] (deltanp1) -- (11.25, .5);
\end{tikzpicture}
\caption{Representation of the computability hierarchy of closed subsets of $\Ms(\azd)$. Arrows indicate strict inclusion relations \cite{Zheng-Weihrauch-2001}.}\label{fig:arith}
\end{center}
\end{figure}

\begin{proposition}[Second computability obstruction]~

Let $F:\azd\to\azd$ be a cellular automaton and $\mu$ be a computable measure. Then $\V(F,\mu)$ is a nonempty $\Pi_2$-computable compact set. 
\end{proposition}

Aiming at a reciprocal, notice that $\Pi_2$-computable compact sets can be all be described as the set of limit points of a sequence of measures $(\meas{w_n})_{n\in\N}$, where the sequence of patterns $(w_n)_{n\in\N}$ is uniformly computable. However, our construction cannot do better that following such a sequence along a polygonal path, that is, along segments of the form $\left[\meas{w_i},\meas{w_{i+1}}\right]=\left\{t\meas{w_i}+(1-t)\meas{w_{i+1}}:t\in[0,1]\right\}$. The following proposition shows that this corresponds to connected limit sets of measures (not necessarily path-connected).

\begin{proposition}[Technical characterisation of $\Pi_2$-CCC sets]\label{prop:polygonalcover}~

Let $\K\subset\Ms(\azd)$ be a non-empty $\Pi_2$-computable, compact, connected set ($\Pi_2$-CCC for short). Then there exists a uniformly computable sequence of cubic patterns $(w_n)_{n\in\N}$ such that $\K$ is the limit of the polygonal path defined by $(w_n)_{n\in\N}$, that is, \[\K=\bigcap_{N>0}\overline{\bigcup_{n\geq N}\left[\meas{w_n},\meas{w_{n+1}}\right]}.\]
\end{proposition}

As we mentioned before, the proof of these statements can be found in \cite{HellouinSablik} or \cite{These} for an extended version.

\section{Construction}
\label{sec:construction}

To obtain the results announced in the introduction,  we prove the following result in conjunction to Proposition~\ref{prop:polygonalcover}.

\begin{theorem}\label{thm:MainTheorem}
For any uniformly computable sequence $(w_n)_{n\in\N}$ of cubic patterns of $\B^\ast$ of dimension at most $d$, there exists a larger alphabet $\A\supset \B$ and a cellular automaton $F : \azd\to\azd$ such that for any Bernoulli measure $\mu\in\Ms(\azd)$: \[\V(F,\mu) = \bigcap_{N>0}\overline{\bigcup_{n\geq N}\left[\meas{w_n},\meas{w_{n+1}}\right]}.\]
\end{theorem}

In the rest of the article, assume some fixed alphabet $\B$ and some uniformly computable sequence $(w_n)_{n\in\N}$ of patterns of $\B^\ast$. We present the construction of the alphabet $\A$ and cellular automaton $F$.


\subsection{Sketch of the construction}

We detail the construction of $\A$ and $F$ by describing the tasks to be performed on the initial configuration. Each letter of $\A$ is a product of seven layers separated in three groups, each group representing some information needed to perform a given task. The alphabet of each layer contains a special \emph{blank} symbol $\Bl$ to denote the absence of information.

\begin{itemize}
  \item The first group is dedicated to the colonising of the configuration. Since we have no control over the contents of the initial configuration, we want to erase (almost) all symbols present initially in favour of various processes that we can control and synchronise. To do this, the \emph{birth layer} contains a \emph{seed} symbol $\sti$ that can only appear in the initial configuration. Each seed gives birth to a stationary \emph{heart} $\stp$ on the same layer, and to a \emph{membrane} on the \emph{growth layer} which grows in every direction. As it grows, the membrane erases everything in its path, except for other membranes issued from a seed with which it merges.
  \item The second group is used to divide the colonised space into mostly independent areas called \emph{organisms}, each organism having at its centre a heart issued from a seed. The borders between organisms are redefined regularly by processes on the \emph{organism layer}. Furthermore, organisms need to grow in size regularly, which is achieved by merging organisms whose hearts are close using the \emph{evolution layer}.
 \item The third group deals with the internal metabolism of the organisms. The goal is first to compute each $w_n$ in succession, which is achieved by simulating a Turing machine in the \emph{computing layer}; then, the main layer of the whole body of the organism is filled with concatenated copies of $w_n$ by using a copying process taking place on the \emph{copying layer}. The above is done synchronously in all organisms, at some time $t_n$ for each $w_n$.
\end{itemize}

Copies of $w_n$ are written on the \emph{main layer}, which implies that the corresponding alphabet is $\B\cup\{\Bl\}$. To sum up, the global alphabet is $\A=\Abirth\times\Agrowth\times\Aorga\times\Aevol\times\Acomp\times\Acopy\times(\B\cup\{\Bl\})$. We check that $\B\subset \A$ up to the bijection $b \mapsto (\Bl,\Bl,\Bl,\Bl,\Bl,\Bl,b)$. Denote $\pbirth$, $\pgrowth$, $\porga$, $\pevol$, $\pcomp$, $\pcopy$, $\pmain$ the projections on each coordinate.


During the description of $F$, we will treat each layer successively. The layers were introduced in order of dependency, in the sense that the time evolution of symbols in a given layer only depends on the contents of layers in the same group and the one immediately preceding it. Furthermore, the main layer is only affected by the copying layer.

\subsection{Space colonisation: Seeds and membranes}

In this section, we describe the cleaning of the configuration through the seeds and the birth, growth and fusion of membranes. We deal only with the birth layer and alphabet $\Abirth$ for the moment. Section~\ref{sec:CreationMyth} gives the general ideas of the process, while Section~\ref{sec:Implementation1} focuses on technical difficulties of the cellular implementation.

\subsubsection{Creation myth: a sketch}\label{sec:CreationMyth}

\paragraph{Seeds and hearts} The alphabet $\Abirth$ contains the seed symbol $\sti$, which can only appear in the initial configuration and cannot be produced by the local rule of $F$. At the first step, each seed spawns a number of processes and turns into a heart $\stp\in\Abirth$. This heart and those processes (and those spawned from them) are called \emph{initialised}, which means that their behaviour is well controlled and synchronised (since they are all born at time $1$). All other symbols are \emph{uninitialised}.

If two seeds are too close from each other ($d_\infty$ less than 5), the largest (in lexicographic order) is erased to give enough space to the other seed to spawn its processes. A seed that is not destroyed at time $1$ in this way is called \emph{viable}. By abuse of notation we write $\pbirth(c_x) = \sti^V$ to mean that the configuration $c$ has a viable seed at coordinate $x$.

\paragraph{Birth of membranes} Each occurrence of $\sti$ triggers the birth at time $1$ of a living \emph{membrane}. The membrane consists in membrane symbols $\memb$ and $\membco$ (and all their rotations) that form initially the surface of an hypercube of edge length $5$ centred on the seed. The membrane is oriented, being able to distinguish inside from outside through orientation vectors. 

\begin{definition}
  A membrane $m$ at time $t$ is a maximal $1$-connected set of coordinates containing membrane symbols $\memb$ or $\membco$ at time $t$ with consistent outward orientation; i.e., orientation of neighbouring membrane symbols differ in at most one coordinate, and at most by $1$.
\end{definition}

When a membrane $m$ forms a closed curve (which is the case for initialised membranes), we denote $\Supp(m)$ its support. In this case, $\Supp(m)$ partitions $\Z^d$ into a finite set $\overline{\Int}(m)$ and an infinite set $\Ext(m)$ which are $\infty$-connected. We also denote $\Int(m)=\overline{\Int(m)}\setminus\Supp(m)$. By "outward" in the previous definition, we mean that the orientation vectors of $m$ are directed towards $\Ext(m)$. 

If a membrane has a malformation that can be detected locally around a symbol (e.g. no neighbours, inconsistent orientations), a death process is spawned. Since uninitialised membranes can be locally well-formed, this is not enough to discriminate them from initialised membranes.

\paragraph{Breathing, growing, getting older} Each membrane symbol is associated with an \emph{age counter}, which is a binary counter initialised at $0$ and increasing at each step, whose aim is to keep track of the value of $t$. Notice that in an initialised membrane all age counters are equal. Figure~\ref{fig:membrane} represents some part of a membrane with arrows and counters.

From time $1$ onward, the membrane grows slowly towards the outside, erasing the content of other layers as it progresses with the exception of other membranes (see next paragraph). This is governed by the \emph{respiration process}: each time the age stored in its counter is the square of an integer, the membrane grows to the outside, making one step in every direction. Technical details related to the implementation of age counters and the respiration process are the object of the next section. 

\begin{figure}[htbp] \centering
  \includegraphics[width=6cm]{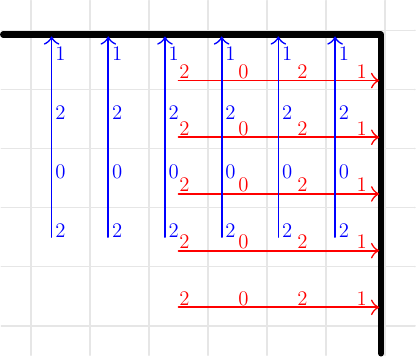}
  \caption{A corner of a membrane in dimension $2$. The arrows give the orientation and the counters store the age of the membrane.}
  \label{fig:membrane}
\end{figure}

 \paragraph{Fight for survival and death} When the growing membranes meet other membrane symbols, they try to determine locally whether they are part of an initialised membrane (in which case the two should merge), or some uninitialised symbols which should be erased. We call \emph{dead} an uninitialised group of membrane symbols that present some locally detectable malformation, such as non-connexity, the absence of or inconsistencies between age counters/respiration processes, inconsistencies between inner and outer orientation for neighbours, etc. In this case, the malformation generates a death signal $\death$ that spreads through the whole membrane erasing it. However, such a group can also form a \emph{zombie membrane}, that is apparently well-formed though uninitialised. Initialised (living) and zombie membranes are distinguished through age counters.

\begin{fact}\label{lem:minimal_age}
At time $t$, all age counters associated with a well-formed membrane have value at least $t-1$, the minimum being reached only for initialised membranes.
\end{fact}
Indeed, age counters of initialised membranes are initialised at $0$ at time $1$, while age counters of zombie membranes were already present (with a positive value) at time $0$, and both are incremented by $1$ at each step. 

\subsubsection{Implementation of age counters}\label{sec:Implementation1}

This section is dedicated to the details of the cellular implementation of the age counters in the growth layer. First we show how to implement age counters with binary counters using logarithmic space. $\Agrowth$ contains membranes symbols and age counter symbols. Each membrane symbol $\memb$ or $\membco$ contains an outward orientation label consisting of a vector of $\mathcal{U}nit(d)$. 


\paragraph{Basis and carry propagation} We use a redundant binary basis in all counters. Let $c = c_{n-1}\dots c_0 \in \{0,1,2,\overline{1}\}^n$ be a counter. The \textbf{value} of $c$ is $\sum_{i=0}^{n-1}c_i2^i$ (reverse order) where $\overline{1}$ has value $-1$. Since $2 = 10$, 2 can be seen as a $0$ with a carry, and $\overline{1}$ as a $0$ with a "negative" carry.

At each time step, carries are propagated along the counter, which can be done in a local manner ($02 \rightarrow 10, 12\rightarrow 20, \Bl2\rightarrow 10, 1\overline{1} \rightarrow 01, 0\overline{1}\rightarrow \overline{1}1$). Additional zeroes at the beginning of the counter are erased ($\Bl1\overline{1} \rightarrow \Bl1$).

In order to increment or decrement the counter by one, which is the case for age counters, the rule is adapted at the least significant bit of the counter (for incrementing: $0 \rightarrow 1, 1\rightarrow 2, 2\rightarrow 1$). 

\paragraph{Age counters} The age counters are implemented as follows. The least significant bit of each counter is next to its corresponding membrane symbol, and the following bits lie on a line directed towards the inside of the membrane. To each possible direction (corresponding to some $\pm e_j$) corresponds a different sublayer, which allows counters to cross near the corners. Thus the age counters use $2d$ sublayers, each sublayer containing symbols $\{\Bl, \overline{1},0,1,2\}$.

Recall that any inconsistency -- the absence of an age counter for some membrane symbols, parallel counters containing different symbols, etc. -- spawns a death process, which spreads in the whole membrane and erases all layers of these cells.

\subsubsection{The respiration process}

The goal of the respiration process is to govern a slow growth of the membrane. 

\paragraph{Breathing counters} Along with the age counter, on two other sublayers of $\Agrowth$, two counters $A$ and $B$ are initialised at time $0$ with values $1$ and $0$, respectively. From there on three phases alternate, the current phase being labelled on the membrane symbol:
\begin{description}
\item[Phase $+$:] $A$ is decremented while $B$ is incremented. When $A$ reaches $0$, the phase passes to $-$;
\item[Phase $-$:] $A$ is incremented while $B$ is decremented. When $B$ reaches $0$, the membrane breathes; 
\item[Breath:] For one step $A$ is incremented while $B$ is unchanged, then the phase passes to $+$.
\end{description}

The value of $A+B$ is constant during a cycle ($1$ for the first cycle) except for the last step where it is incremented by one. The cycle takes a total time $2(A+B)+1$. Therefore a breath occurs at each time $t^2$ for $t>1$. In Figure~\ref{fig:UpdateCounters} we represent the update operation of all three counters, that is, incrementing the age and updating $A$ and $B$ according to the phase.
 
 \begin{figure}[h]
\begin{center}
 $  \begin{array}{c|c|c|c|c|c|c}
 \quad t\quad &\ phase\ & \quad A\quad &\quad B\quad &\quad age\quad & 5+2\lfloor\sqrt{t}\rfloor\\
 \hline
 \hline
 1&+&1&0&0&5\\
 2&-&0&1&1&5\\
 3&-&1&0&2&5\\
 4&breath&2&0&11&7\\
 5&+&1\overline{1}&1&12&7\\
 6&-&0&2&21&7\\
 7&-&1&1\overline{1}&102&7\\
 8&-&2&0&111&7\\
 9&breath&11&0&112&9\\
 10&+&10&1&121&9
 \end{array}$
   \caption{ Values of the three counters for $t\leq 10$.} \label{fig:UpdateCounters}
   \end{center}
 \end{figure}

 \begin{lemma}
   \label{lem:counter_update}
   The counter update can be performed locally with radius $2$.
 \end{lemma}
 \begin{proof}
   The incrementations and decrementations described above can be achieved with radius $1$. The least significant bit can be distinguished by being next to the membrane symbol which contains the information on the current phase and the most significant bit is next to a blank symbol (on its layer).

We show that the fact that a counter is worth $0$ is detectable with radius $2$ (to see why this is nontrivial, consider the update $\Bl{1}\overline{1}\overline{1}\dots \overline{1} \rightarrow \Bl 0 0\dots 0$). During a decrementation the least significant bit alterns between $0$ and $\overline{1}$. Since carries progress at "speed" one, two negative carries can never be next to each other. Therefore the only possible representations of $1$ are $\Bl{1}$ and $\Bl1\overline{1}$, and both yield $\Bl 0$ at the next step. Therefore detecting when the counter is worth $0$ requires radius two, in order for the membrane symbol to "see" the word $\Bl 0$.
%
\qed\end{proof}

\paragraph{Breathing process} Each breath makes the membrane progress by one cell on every direction, a process we detail below. If such a symbol is not produced synchronously by the whole membrane, then $\death$ signals spawn and spread to erase the membrane. 

Recall that each membrane symbol at coordinates $x$ is labelled with an outward growth direction, which is a vector $v=\sum\varepsilon_je_j\in\mathcal{U}nit$. The membrane symbol is the border between $\Int(m)$ and $\Ext(m)$, with the orientation vector indicating which part $\Ext(m)$ is. When a breath symbol appears, the membrane symbol is removed and new symbols are created on all cells of $\Ext(m)$ that were $\infty$-adjacent to any membrane symbol of $m$. The new orientation vectors are determined by remaining coherent with the old orientation.

The remaining task is to reproduce the counters for the new symbols.

First consider the case of a face symbol $x$ and orientation $e_j$. Right after a breath, when a new symbol is created at coordinate $x+e_j$, the symbol at $x$ is replaced by a placeholder symbol $s_{lim}$. This symbol progressively shifts the counters of $x$ by one cell in direction $e_j$, marking at each step the limit between the part which is to be shifted and the part already shifted. The counters keep updating by ignoring this symbol, which increases the radius to 3. Figure~\ref{fig:shift} illustrates the shift.

 \begin{figure}[htbp] \centering
  \includegraphics[width=6cm]{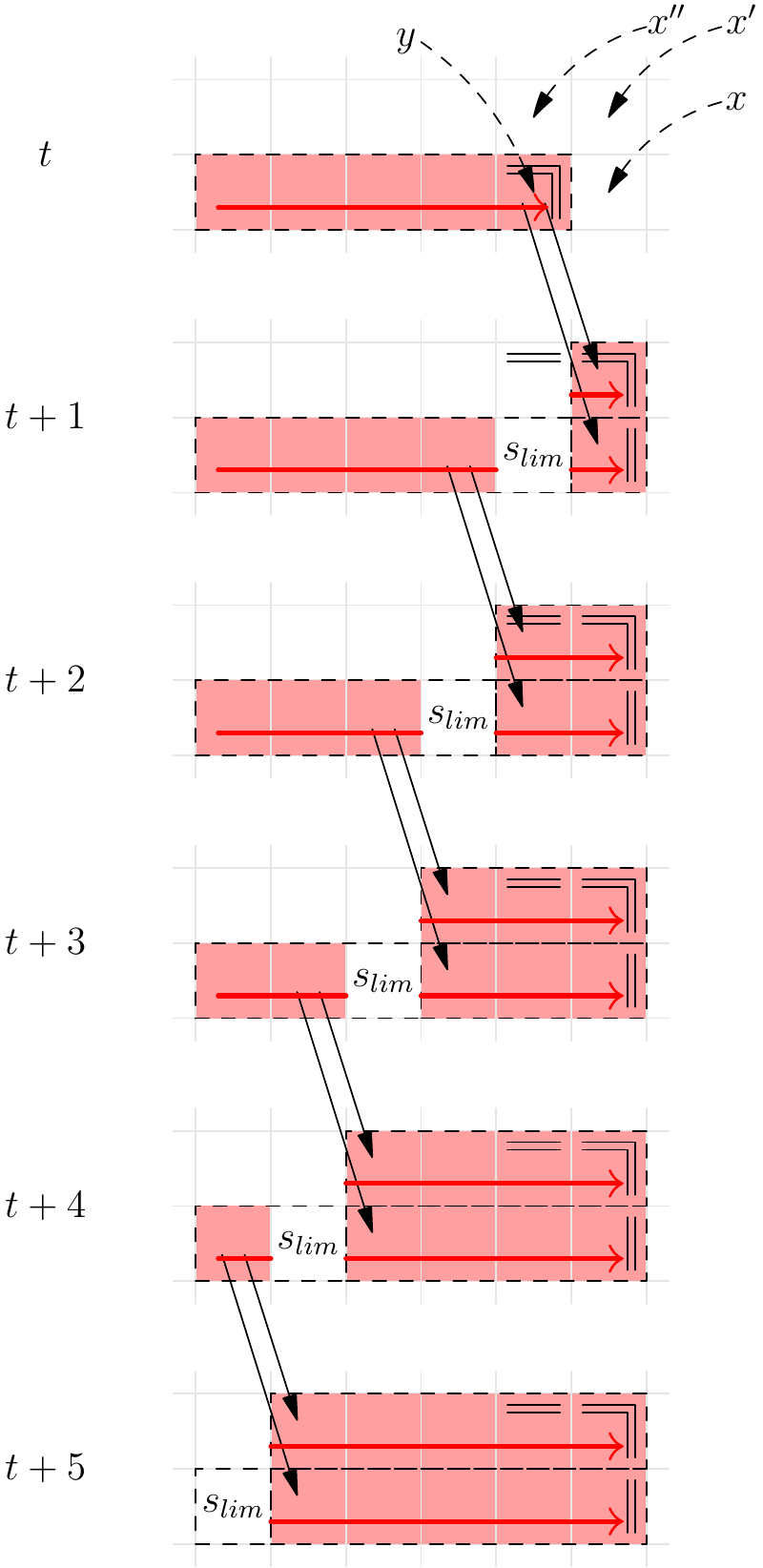}
  \caption{After a breath, a membrane corner symbol at cell $y$ is erased and new membrane symbols appear at $x$, $x'$ and $x''$.  For the sake of readability, the vertical counters are not represented, hence we draw only the new horizontal counters in $x$ and $x'$. At each step between $t+1$ and $t+4$, the red cells represent superposition of the age, $A$ and $B$ counters. They are copied  to the new membrane symbols, but the incrementation does not stop.}
  \label{fig:shift}
\end{figure}

From this section we deduce the following:
\begin{fact}
Each initialised membrane forms an hypercube of edge length $5+2\lfloor \sqrt{t}\rfloor$ at time $t$. 
\end{fact}
Since the counters of an initialised membrane are initialised to $0$ at time $1$, the maximum size of the counters is $\lceil\log_2(t-1)\rceil$ at time $t$ and membranes have enough space to contain them. For technical reasons that will become clear later, we need to quantify the maximal number of breaths of any membrane symbol (not necessarily initialised) in a given time:
 
\begin{lemma}
  \label{lem:density_extensions}
 The number of breaths triggered by any membrane symbol between times $t$ and $t+k$ is at most $\lfloor\sqrt{t+k}\rfloor-\lfloor \sqrt{t}\rfloor$.

 \end{lemma}
 \begin{proof}
Apart from time $0$ (when a breath symbol could be present), a breath is only triggered when the B counter of a membrane symbol without local malformations reaches $0$. This symbol must be issued from a seed or from a membrane symbol already present at time $0$. In the first case, since a breath is triggered at each step when the time $t$ is the square of an integer (except for $1$), the number of breaths before time $t$ is $\lfloor\sqrt{t}\rfloor-1$. The lemma follows.

In the second case, the membrane symbol had at time $0$ counters $A$ and $B$ with some positive values $a_0$ and $b_0$ and some phase $\varepsilon_0$, values which correspond to those of an initialised set of counters at some time $t_0>0$. From there on the time evolution of the membrane symbol is similar to the evolution of an initialised membrane symbol of age $t_0+t$, which means that the number of breaths between times $t$ and $t+k$ is $\lfloor\sqrt{t+t_0+k}\rfloor - \lfloor\sqrt{t+t_0}\rfloor \leq \lfloor\sqrt{t+k}\rfloor-\lfloor \sqrt{t}\rfloor$.
%
\qed
 \end{proof}

 \paragraph{Forming colonies} For an initial configuration $c$, define $M_t(c)$ to be the set of initialised membranes at time $t$. Then the \emph{colonised space} at time $t$ is:
\[\Col_t(s)=\bigcup_{m\in M_t(c)}\overline{\Int}(m).\] When a membrane grows, it erases the content of every other layer of the cells it encounters, except when the birth layer contains the outer border of a membrane. In this case, the comparison process starts, which is the topic of the next section. 

To sum up, the alphabet $\Abirth$ contains seeds and hearts, and $\Agrowth$ contains the states used for membranes (including counter sublayers). As we will see in the next section, $2^{d}$ different membranes can share the same cell, so this alphabet will be duplicated this many times.
  
\subsubsection{Survival of the youngest}
As membranes grow and tend to cover the whole space, different membranes eventually meet. The result of the encounter should depend on the nature of the membranes: two initialised membranes should merge while an initialised membrane should erase an uninitialised membrane (what happens between uninitialised membrane is irrelevant). In this section, we devise a comparison process to distinguish initialised from uninitialised membranes, using the growth layer and its alphabet $\Agrowth$.

From Fact~\ref{lem:minimal_age} we know that initialised membranes have the youngest age counters, and only tie with other initialised membranes. The value of the age counters are compared to let the younger membrane survive, with merging occurring in case of equality.

\paragraph{Many membranes meeting} When we say that two membranes $m$ and $m'$ meet at time $t$ in cells $x$ and $x'$, we mean that there exists $1\leq j\leq d$ such that:
\begin{itemize}
\item either $x\in\Supp(m)\cap \Supp(m')$, $x+e_j\in\Ext(m)$ and $x-e_j\in\Ext(m')$, in which case take $x'=x$;
\item or $x\in\Supp(m)\cap \Ext(m')$ and $x+e_j\in\Supp(m')\cap\Ext(m)$, in which case take $x'=x+e_j$.
\end{itemize}
The two possible situations are illustrated in Figure~\ref{fig:meeting}.

\begin{figure}[htbp] \centering
  \begin{tabular}{ccc}
    \includegraphics[height=4cm]{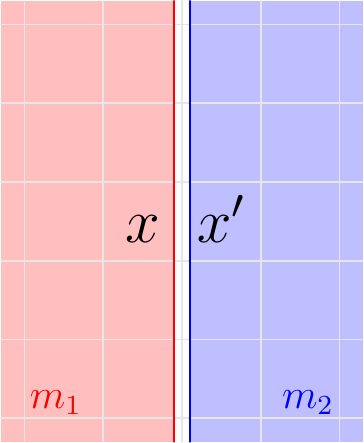} &
    \hspace{2cm}&
    \includegraphics[height=4cm]{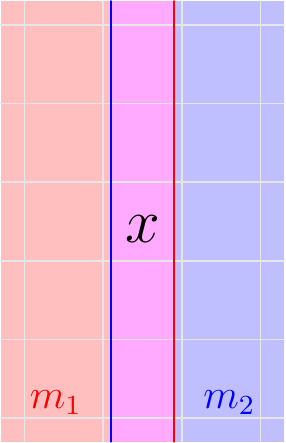}\\
    \end{tabular}
  \caption{Depending on the parity of the distance between membranes, they will meet either when they share some border cells ($x'=x$), or when the borders are adjacent($x'=x+e_j$).}
  \label{fig:meeting}
\end{figure}

In particular, the membranes arriving from opposite directions, they have (at least) an age counter in opposite directions, say $e_j$ and $-e_j$. These age counters are to be copied on a dedicated sublayer of $\Agrowth$ and compared. 

\paragraph{Copying phase} At positions $x$ and $x'$, two symbols $\cun$ and $\cunde$ are written on the growth layer to trigger the process (if $x=x'$, a symbol $\cundemi$ represents the superposition of those symbols) and progress at speed one in the corresponding direction, copying at each step one bit from the age counter to a sublayer of $\Agrowth$. Carries $2$ are copied as $0$: indeed, the copy is performed at the same speed as the carry progresses, so it would be copied at each step otherwise (the carry is taken into account once it turns a $0$ into a $1$). More generally, only carries that appeared before the beginning of the copy can influence the copied counter, which is not incremented. Thus the copied counter have the same value as the age counter at the beginning of the copy. When it reaches the end of its counter, each copy symbol turns into a comparison symbol $\cde$ (resp. $\cdede$), which triggers the comparison phase. 

\paragraph{Comparison phase} The comparison symbols return towards the meeting point, "pushing" the copied bits in front of them in a caterpillar-like movement, starting from the most significant bit. The returning bits use another sublayer of $\Agrowth$. The process is represented in Figure~\ref{fig:compar}.

\begin{figure}[htbp] \centering
 
    \includegraphics[width=8cm]{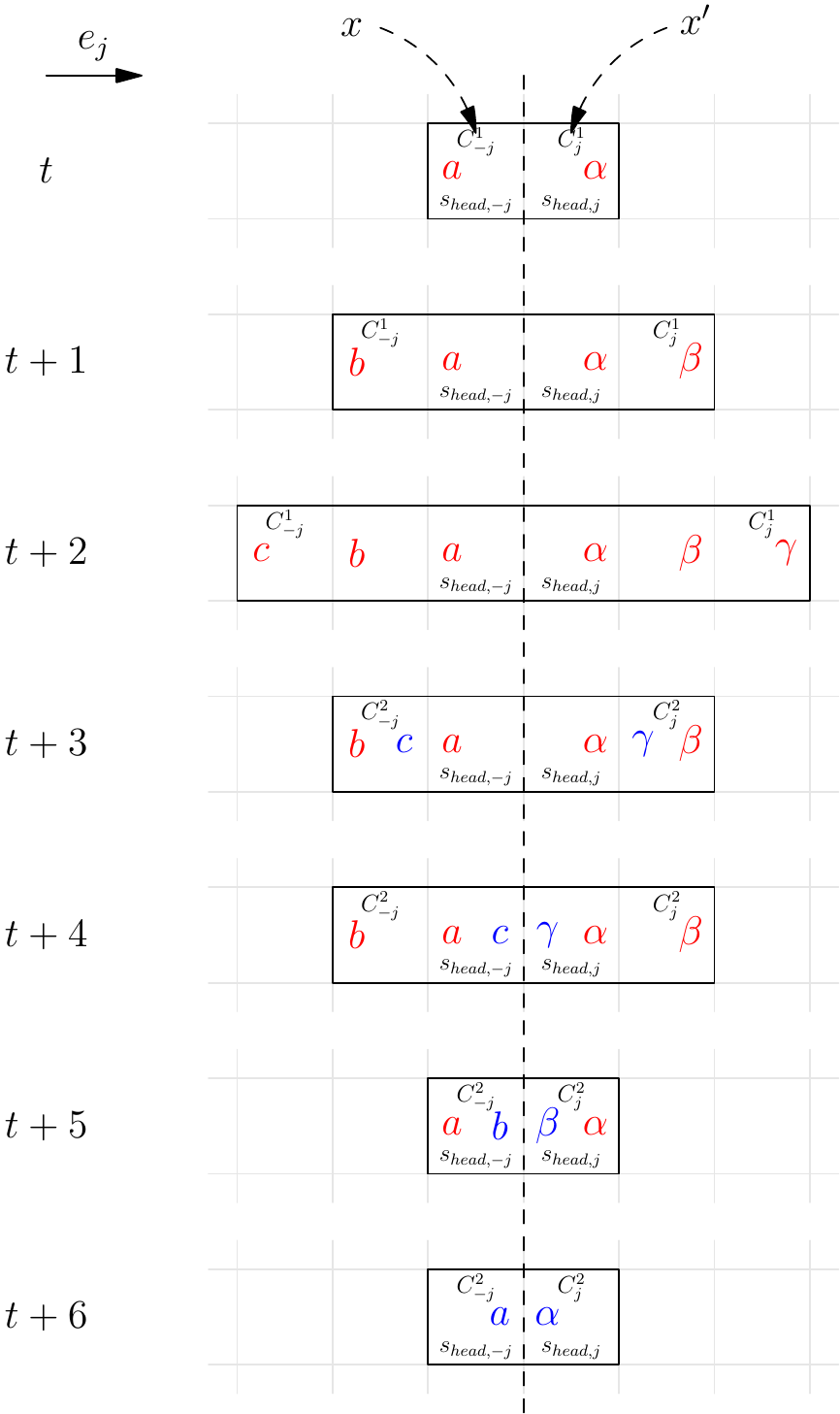}
  \caption{In this example, two membranes meet in cells $x$ and $x'$ at time $t$. Their age counters are respectively $abc$ and $\alpha\beta\gamma$ of same length $3$. Only the growth layer is represented. At the end ($t+6$), the decision can be made in both $x$ and $x'$. In this particular case, the symbols $s_{head}$ have not been moved, which means neither of the membranes did extend during the comparison.}
  \label{fig:compar}
\end{figure}
 
 As the returning bits reach the meeting point, one of the following situations occur:
 \begin{itemize}
 \item the most significant bit from one side arrives earlier than the most significant bit from the other side. In this case the age counter of the corresponding side is shorter, which means that the membrane of this side is younger;
 \item both most significant bits arrive simultaneously at the meeting points $x$ and $x'$. Then bits are compared as they arrive. The first bit that is smaller than its counterpart corresponds to the side of the younger membrane;
 \item in the previous case, if all bits are equal until the end, both membranes have exactly the same age.
 \end{itemize}

Those three possibilities are tested locally at the meeting point and the result is written as a symbol (on its own sublayer) marking the direction of the younger membrane, with $=$ in case of a tie. If for some reason a symbol $\death$ (death process) reaches the symbol of one of the sides, the comparison stops and the surviving membrane is marked as younger "by default".

If a membrane is declared younger, all auxiliary symbols used for comparison are erased and a death process triggers in the older membrane. The younger membrane will resume its growth naturally at the next breath. If the result is a tie, both membrane symbols are erased along with all associated auxiliary states: the membranes are merged. 

\begin{remark}\label{rem:comparison}~
   
   \begin{itemize}
   \item In general two membranes may have more than one meeting point. In that case, comparisons are performed simultaneously at every point and in every concerned direction. In the case of a tie, all symbols participating in the meeting are erased simultaneously; any local discrepancy results in the spawn of a death process.

   \item At most $2^d$ different membranes can meet in the same cell $x$ (corners of hypercubes arriving from all possible directions). To solve this problem, we duplicate each sublayer (in $\Abirth$ and $\Agrowth$) used in the comparison into $2^d$ copies, each copy being able to perform a comparison independently of the others. If the membrane is older than at least another membrane, a death process is spawned; it merges if it is tied for youngest.
   \item Let $\ell$ be the length of the shortest age counter. The previous process needs $\ell$ steps to copy this age counter on the growth layer, and $2\ell$ steps to send them one by one to the meeting point. Regardless of the length of the other counter, the comparison reaches a result after the last bit of the shortest counter arrives. Therefore the whole process takes at most $3\ell$ steps. Remember that $\ell = \lceil\log{t-1}\rceil$ if one of the membranes is initialised (assuming the comparison process began at time $t$).
   \end{itemize}
\end{remark}

\paragraph{Breathing during comparisons} We did not take into account the possibility that one of the membranes breathes during the comparison. For each meeting of a pair of membranes, call \emph{instigating membrane} the one whose breath has triggered the meeting (possibly both if they moved simultaneously; this is the case for initialised membranes).
 
\begin{lemma}
  \label{lem:extensions}
Let $m$ be a living membrane meeting another membrane at time $t$. During the comparison process, $m$ may breath at most one time if it is not instigating, and cannot breath at all if it is instigating.

\end{lemma}
\begin{proof}
  Using the above remark, we know that the comparison process takes at most $k=3\lceil\log{(t-1)}\rceil$ steps. Using Lemma~\ref{lem:density_extensions}, the number of breaths of $m$ between times $t-1$ and $t+k$ is at most:
  \begin{align*}
    \left\lfloor\sqrt{t+k}\right\rfloor-\left\lfloor \sqrt{t-1}\right\rfloor
    &\leq \left\lceil\sqrt{t+k}- \sqrt{t-1}\right\rceil\\
    &\leq \left\lceil \frac{k+1}{2\sqrt{t-1}}\right\rceil\quad (\text{since the derivative of }\sqrt{t-1}\text{ is decreasing})\\
    &\leq \left\lceil \frac{3\log{(t-1)+1}}{2\sqrt{t-1}}\right\rceil \leq 1.
    \end{align*}
If $m$ is instigating, then by definition $m$ breathed at time $t-1$ and cannot breath again before time $t+k$. Otherwise, $m$ breathes at most one time.
\qed\end{proof}

Therefore, if one or both membranes move during the comparison because of the respiration process, it writes a head symbol $s_{head,j+}$ to recall its new position. If a membrane extends more than twice before the end of the comparison, a death process is triggered for this membrane. As the radius of $F$ is more than $2$, the moved membrane can still read the result of the comparison.

\begin{lemma}
  \label{lem:life_supremacy}
Take any $t>0$ and initial configuration $c\in\azd$. Then:
\[\Col_t(c)=\{x\in\Z^d:\exists y\in\Z^d, \dinf(x,y)\leq 1+\sqrt{t}, \pbirth(c_y)=\sti^V\}.\]
In other words, the colonised space at time $t$ is exactly the set of cells that, at time $0$, are at distance less than $1+\sqrt{t}$ from a viable seed.
\end{lemma}
\begin{proof}
We prove this result by structural induction. If $t=1$, then the colonised space is the set of all initialised membranes which are hypercubes of edge side $5$ around each viable seed, and the result is proved.

Now suppose that the hypothesis holds at time $t$. Notice than $\Col_t(c)\subset \Col_{t+1}(c)$, and that merging does not add any cell to the colonised space: an initialised membrane cannot be erased, and the colony obtained after merging two membranes is the union of the colonies defined by the two merged membranes. Only the breathing process may add new cells to the colonised space.

Consequently, the induction step is empty if $(t+1)$ is not a square since $\Col_t(c) = \Col_{t+1}(c)$ and $\dinf(x,y)\leq 1+\sqrt{t} \Leftrightarrow \dinf(x,y)\leq 1+\sqrt{t+1}$ (distances are integers). If $(t+1)$ is a square, then all membrane symbols in initialised membranes take a breath and extend by one cell towards the outside. Now, take a cell $y$ at distance $1+\sqrt{t+1}$ from the nearest viable seed. By the induction hypothesis, $y\notin \Col_t(c)$, but $y$ has a neighbour $y+v$ with $v\in\mathcal{U}nit$ at distance $1+\sqrt{t}$ from that seed, so that $y+v\in\Col_t(c)$. Therefore $y+v$ must be a membrane symbol in the support of an initialised membrane that breathes at time $t+1$, and therefore $y\in \Col_{t+1}(c)$. Conversely, if a cell $z$ is at distance greater than $1+\sqrt{t+1}$ from the nearest viable seed, it cannot have a membrane symbol belonging to an initialised membrane as a neighbour, so that $z\notin \Col_{t+1}(c)$.

\qed\end{proof}

\subsection{Colonies: evolution of the population}
\label{sub:colonies}
 From Lemma~\ref{lem:life_supremacy} we can see that only the contents of the colonised space matter asymptotically. In this section, we describe the interaction of organisms inside colonies. In all the following, we assume we are inside a colony, and the support of the surrounding membrane acts as an impassable wall for any symbol in the second group layers: the organism and evolution layers.

\subsubsection{Hearts and organisms}

As we saw, seeds $\sti$ at time $1$ spawn a membrane and turn into hearts $\stp$. Each heart will be the centre of an \emph{organism} which is itself a subset of the colony. At first each colony have only one heart, but as initialised membranes merge together, the colonies become multi-hearted, and the colony space is partitioned into organisms (except possibly a negligible part). For various reasons, the size of the organisms should grow in a controlled way, which requires some hearts to be progressively removed.

In the present section, we present the cycle of division of colony space and life of the organisms, and all symbols presented here belongs to $\Aorga$.

The life of an organism consists in a succession of generations. We introduce a sequence of times $(t_n)_{n\geq 1}$ (to be fixed later), marking the limit between the $n-1$-th and $n$-th generation. Time is tracked by the heart through a binary \emph{time counter} similar to age counters, initialised at $1$ at time $1$ (along with the heart) and remaining stationary next to the heart. Details on the implementation and the way to determine when $t=t_n$ will be given in Section~\ref{sec:Computing}.

\paragraph{Organisms expanding} At each time $t_n$, \emph{organism-building signals} spread from every heart, progressing as membrane symbols but with speed $1$ (although they do not carry any counters). While progressing, they erase the old contents of the second and third group layers except for the main layer. When they meet a membrane or another organism-building signal, they vanish leaving behind a \emph{neutral border symbol} $\std$. For parity reasons, if two signals emitted by hearts in $x$ and $x'$ arrive simultaneously in two neighbour cells $y$ and $y'$, they leave behind two \emph{pseudo-border symbol} $\stdp$ with an orientation vector towards the interior of their organism, that is, the direction opposite to the initial organism-building signal. 
Just as membrane symbols, $3^d-1$ different organism-building symbols and pseudo border symbols are required (one for each orientation).

The \emph{territory} of a heart is the maximal set of $1$-connected cells containing the heart and no neutral border symbol $\std$ nor pseudo border symbol pointing towards another organism; in other words, the set of cells reached first by organism-building signals emitted by this heart. At time $t_n+k$ (assuming $t_n+k<t_{n+1}$), the only cells of the colony that are not part of some organism are either at distance more than $k$ from the nearest heart, or were outside the membrane at time $t_n$ (and a breath had included them since). 

\begin{fact}
Let $x$ be a cell containing a heart at time $t$ with $t_n\leq t\leq t_{n+1}$, and let $y$ be a cell in its territory. Then the organism-building signal emitted by $x$ reached $y$ at time $t_n+\dinf(x,y)$ and no other organism-building signal reached a neighbour of $y$ before that time.
\end{fact}
%


The following lemma gives insight about the shape of the territories, namely, that they are a (discrete) star domain, whether the borders are included or not.

\begin{lemma}
\label{lem:triangul}
If a cell $y$ belongs to the territory of a heart in cell $x$, then each cell $y'$ such that $\dinf(x,y')+\dinf(y',y)=\dinf(x,y)$ is also in this territory. Furthermore, $y'$ can be a border only if $y$ is a border.
\end{lemma}

\begin{proof}
  For the first part of the lemma, suppose such a $y'$ is not in the territory of $x$. We can build an $\infty$-path between $y'$ and $y$ consisting of cells $(y^{(i)})_{0\leq i\leq N}$ such that $\dinf(x,y^{(i)})+\dinf(y^{(i)},y)=\dinf(x,y)$. Take $y^{(j)}$ the first $y^{(i)}$ that belongs to the territory of $x$. 
  
  Denote $T=t_n+\dinf(x,y^{(j-1)})$ the time when the organism-building signal emitted by $x$ should have reached $y^{(j-1)}$ in the absence of any other heart. Since $y^{(j-1)}$ is not in the territory of $x$, there must exist another heart $x'$ that emitted an organism-building signal that arrived in $y^{(j-1)}$ before time $T$ (recall that the pseudo-borders are considered as parts of organisms). 
  
  Since $y^{(j)}$ is adjacent to $y^{(j-1)}$, $y^{(j)}$ is reached by some signal before time $T+1$. But the signal from $x$ cannot reach $y^{(j)}$ before time $t_n+\dinf(x,y^{(j)}) = T+1$. Therefore $y^{(j)}$ is not in the territory of $x$, a contradiction.

For the second case, notice that $y'$ is a border if and only if both signals reached this cell simultaneously. Then the same reasoning along the path $(y^{(i)})_{0\leq i\leq N}$ shows that these cells cannot be inside the territory of $x$.
\qed\end{proof}

\subsubsection{Natural selection}

In this section we consider the evolution layer and the alphabet $\Aevol$.

To have enough computation space and ensure that the auxiliary symbols are in negligible density, the minimal size of the organisms should grow regularly. More precisely, we require that the territory of any organism during the $n$-th generation contains at least a hypercube of side length $2n+1$ centred at its heart. If two hearts are at distance less than $2n+1$, they are said to be in \emph{conflict}. In this section, we devise a selection process to detect this fact and to erase one of them.

The first lemma is related to the quantity and position of hearts conflicting with a given heart. 
\begin{definition}
For each point $x$ and each vector $\mathcal R \in \{>, <, =\}^d$ that is not $=^d$ we introduce the corresponding \emph{quadrant} defined as $\{y\in\Z^d\ :\ \forall i,\ y_i\mathcal R_i x_i\}$. 
 
For $n\in\N$, each quadrant contains a unique \emph{$n$-extremal point}, which is a point $y$ such that $y_i-x_i\in\{-n, n, 0\}$ for all $i$. Denote $Ext_n(y)=\{\varepsilon e_i: \varepsilon\in\{1,-1\},y_i-x_i=n\varepsilon\}$ the set of its directions of extremality relative to $x$. Notice that this notion is defined for any vector $y$ such that $d_\infty(x,y)=n$.
\end{definition}

\begin{lemma}
  \label{lem:conflict_sim}
At any given generation, a heart conflicts with at most one other heart in each quadrant. In particular, it conflicts with at most $3^d-1$ hearts.
\end{lemma}

\begin{proof}
Let $x$ be the position of the heart and $t_n\leq t<t_{n+1}$ the current time. The heart in $x$ conflicts with another heart at $y$ if and only if $d_\infty(x,y)\in\{2n,2n+1\}$. If two conflicting hearts were in the same quadrant, then they would be at distance at most $2n-1$, a contradiction. The bound is tight and reached by the configuration where there is a heart per quadrant at each $2n$-extremal point relative to the heart in $x$.
\qed\end{proof}

\paragraph{Settling conflicts at random} Each heart keeps in memory a random \emph{central bit} and the winner of each conflict is decided through the value of these bits: if the bits are equal, destroy the largest heart (in lexicographic order), and the smallest otherwise. As long as the central bits of all hearts are independent of each other, this process will ensure that the probability of each heart to be destroyed in any conflict has a constant lower bound that only depends on the initial Bernoulli.

\begin{remark}
In dimension $1$~\cite{HellouinSablik,bdpst}, the heart in the smallest cell was systematically killed. The reason why we adopt a more sophisticated kill choice method, similar to the one used in~\cite{Delacourt-Hellouin}, is that we later need to control the growth rate of the organisms. This growth rate analysis under the simpler method proved to be more difficult than in dimension $1$ and could not be carried out. 
\end{remark}

\paragraph{Independent bit harvesting} In addition to its central bit, each heart maintains a collection of $3^d-1$ \emph{side} bits, one for each quadrant, that must remain independent of each other and of every other central or side bits of other hearts. This is achieved by using the independence of the initial measure. At time $0$, each seed looks at its $3^d-1$ adjacent cells (which are in one-to-one correspondence with its quadrants): if a cell contains the symbol $\memb$ (the symbol choice is arbitrary) then the corresponding bit in the newborn heart is put to $1$, and $0$ otherwise. Remember that if two seeds are too close, one is erased, hence those bits are really independent from one another and from other hearts. 

Each time a heart kills another in a conflict, the dying heart transmits one of his side bits to its killer (the bit corresponding to which quadrant the killer belongs to). The new central bit of the killer is the sum of the side bits received from its victims. As we will see, this process lets us maintain probability bounds on the value of central bits while preserving independence.

\paragraph{Body building with bits} In order to detect conflicting hearts, each heart in position $x$ builds at time $t_n$ a \emph{body}: a hypercube of side length $5$ centred around the heart. Each body symbol in position $y$ carries various pieces of information: the central bit of the heart, which quadrant (relative to the heart) it belongs to, the corresponding side bit, and the set of directions in which it is $2$-extremal ($Ext_2(y)$). Next to the heart, a dedicated process writes two copies of $n$ as binary counters from the computation layer (see next section).

The three phases of conflict resolution are:
\begin{description}
\item[Body building] Each heart sends $n$ (hypercube-shaped) signals that progress at speed $1$ in every direction, one every $n$ steps. The count is kept by decrementing the binary counters. As each signal reaches the body, it pushes it outwards by one cell. The set of directions in which it is $k$-extremal (where $k$ is the number of steps) is kept updated.
  \item[Conflicting] If two bodies intersect, the corresponding hearts are in conflict. Body symbols at $n$-extremal points determine locally the relative positions (quadrants) of the two hearts (see next paragraph), and using the values of the central bits, the winner of the conflict. This phase takes only one step.
\item[Body shrinking] Each heart keeps sending (hypercube-shaped) signals every $n$ steps, but these signals now pull the body inwards as they reach it. The body is destroyed as it reaches size $5$, transmitting its information to the heart. The heart stops sending signals.
\end{description}

If a heart has fought no conflict during a generation, nothing happens. If it loses at least one conflict, it self-destroys. Not being able to send border-building signals, its territory will be occupied by other organisms at the next phase. If it wins all its conflicts, it replaces its central bit by the sum (modulo $2$) of all the bits it receives from its victims.



\paragraph{Quadrant determination} During the conflicting phase, only the extremal points of its body settle the conflict. Take an extremal point of some body in position $y$ such that there is a body symbol in position $y'$ (not necessarily extremal), belonging to another body. Denote $\mathcal R$ and $\mathcal R'$ the respective quadrants of $y$ and $y'$. The extremal point $y$ deals with the conflict if and only if these properties hold:
\begin{itemize}
  \item (I) $y'=y+\sum_{v\in D}v,\ D\subseteq Ext_n(y)$ (note that $D$ may be empty, that is, $y'=y$) ;
\item (II) $\forall \varepsilon e_i\in Ext_n(y)\cap Ext_n(y'),\ y_i\neq y'_i$;
\item (III) for each $i$, if $\mathcal{R}_i$ is $=$ then $\mathcal{R'}_i$ is $=$ as well.
\end{itemize}

The first condition means that extremal points are only interested in directions (or sets of directions) they are extremal in. Both other conditions are illustrated in Fig~\ref{fig:quadrants}.

\begin{lemma}
A conflict is settled exactly once by each involved heart, at the extremal point corresponding to the relative position (quadrant) of the other involved heart.
\end{lemma}
\begin{proof}
Consider some heart in position $x$, and a conflicting heart $x'$ located in quadrant $\mathcal R$. Since $d(x,x')\in\{2n,2n+1\}$, we can check that $d(y,x')\leq n+1$ where $y$ is the $n$-extremal point for $x$ in quadrant $\mathcal R$.  We check that $y$ satisfies (I)  with some neighbour $y'$ and they satisfy both (II) and (III). 

 Since $d(x,x') = 2n+\alpha$ for some $\alpha\in\{0,1\}$, there is some $1\leq i\leq d$ such that $|x_i-x'_i| = 2n+\alpha$. For any such $i$, fix $y'_i-x'_i = -(y_i-x_i)\in\{-n,n\}$ so that $|x_i-y'_i|= n+\alpha$.

For all others $1\leq i\leq d$:
\begin{itemize}
\item if $\mathcal R_i$ is $>$ take $y'_i=x_i+n = x'_i+(n-x'_i+x_i) \in [x'_i-n, x'_i+n]$
\item if $\mathcal R_i$ is $<$ take $y'_i=x_i-n = x'_i-(n+x'_i-x_i) \in [x'_i-n, x'_i+n]$
\item if $\mathcal R_i$ is $=$ take $y'_i= x'_i$.
\end{itemize}

 Hence $d(x',y') = n$ and $y'$ belongs to the body of $x'$.

 (I) For every $1\leq i\leq d$ such that $|x_i-x'_i| = 2n+\alpha$, $|y'_i-y_i|=|(y'_i-x'_i)+(x'_i-x_i)+(x_i-y_i)|=\alpha$. For all others $1\leq i\leq d$:
\begin{itemize}
\item if $\mathcal R_i$ is $>$, then $y'_i-y_i=(x_i+n)-(x_i+n)= 0$.
\item if $\mathcal R_i$ is $<$, then $y'_i-y_i=(x_i-n)-(x_i-n)=0$.
\item if $\mathcal R_i$ is $=$, then $y'_i-y_i=(y'_i-x'_i)+(x'_i-x_i)+(x_i-y_i)=0+0+0$.
\end{itemize}
This proves that $d(y,y')\leq 1$ and (I) is verified.

(II) Assume $y$ and $y'$ share an extremality direction $e_i$ and $y_i=y'_i$. That is, $y_i-x_i = n$ and $y'_i-x'_i = n$, so that $x_i = x'_i$. Furthermore, notice that $\mathcal R_i$ is $>$ since $y$ belongs to quadrant $\mathcal R$. But since $x_i = x'_i$, $x'$ wouldn't be located in quadrant $\mathcal R$, a contradiction. The same argument works for an extremality direction $-e_i$.

(III) If $\mathcal R_i$ is $=$, since $x'$ is located in quadrant $\mathcal R$ relative to $x$, this means that $x_i = x'_i$. Since $y$ is extremal for $x$ in quadrant $\mathcal R$, we have $y_i = x_i$. Therefore $y'$ is located relative to $x'$ in a quadrant $\mathcal R'$ such that $y'_i = x'_i$, i.e. $\mathcal R'_i$ is $=$.

We show that no other extremal point handles the conflict. Let $z$ be another extremal point in quadrant $\mathcal R"$ for $x$, next to a body symbol $z'$ for $x'$. We distinguish various cases:

\begin{itemize}
\item For some $i$, assume $\mathcal R_i$ is $=$ but not $\mathcal R"_i$. Then $x_i = x'_i$ but $z_i = x_i\pm n$, i.e. $z$ has an extremality direction along $e_i$. Similarly $z'$ has an extremality direction along $e_i$ and $z_i=z'_i$, contradicting condition (II).
\item For some $i$, assume $\mathcal R_i$ is $<$ but $\mathcal R"_i$ is $>$. Then $x'_i < x_i$ and $z_i = x_i+n$, so that $z_i>x'_i+n\geq z'_i$ and $z$ does not handle the conflict according to (I).
\item For some $i$, assume $\mathcal R_i$ is $<$ but $\mathcal R"_i$ is $=$. Then $z_i = x_i > x'_i$ and as $\varepsilon e_i\notin Ext_n(z)$, (I) ensures that $z_i=z'_i$. Hence $z'_i > x'_i$ and  $\mathcal R"_i$ is not $=$, which contradicts condition (III).
\end{itemize}
\qed\end{proof}

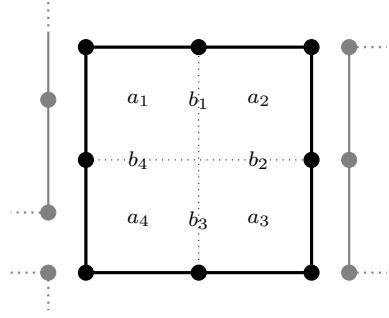
\begin{figure}
\begin{center}
\begin{tikzpicture}
\draw[very thick] (0,0) -- (0,3) -- (3,3) -- (3,0) -- (0,0);
\draw[dotted] (0,1.5) -- (3,1.5) (1.5,0) -- (1.5,3);
\node at (0.7,2.3) {$a_1$};
\node at (2.3,2.3) {$a_2$};
\node at (2.3,0.7) {$a_3$};
\node at (0.7,0.7) {$a_4$};
\node at (1.5,2.3) {$b_1$};
\node at (2.3,1.5) {$b_2$};
\node at (1.5,0.7) {$b_3$};
\node at (0.7,1.5) {$b_4$};
\filldraw (0,0) circle (0.1cm);
\filldraw (1.5,0) circle (0.1cm);
\filldraw (3,0) circle (0.1cm);
\filldraw (3,1.5) circle (0.1cm);
\filldraw (3,3) circle (0.1cm);
\filldraw (1.5,3) circle (0.1cm);
\filldraw (0,3) circle (0.1cm);
\filldraw (0,1.5) circle (0.1cm);

\draw[thick, gray] (-0.5,0.8) -- (-0.5,3.2);
\draw[thick, gray, dotted] (-0.5,0.8) -- (-1,0.8);
\draw[thick, gray, dotted] (-0.5,3.6) -- (-0.5,3.2);
\filldraw[gray] (-0.5,0.8) circle (0.1cm);
\filldraw[gray] (-0.5,2.3) circle (0.1cm);

\draw[thick, gray] (3.5,0) -- (3.5,3);
\draw[thick, gray, dotted] (3.5,0) -- (4,0);
\draw[thick, gray, dotted] (3.5,3) -- (4,3);
\filldraw[gray] (3.5,0) circle (0.1cm);
\filldraw[gray] (3.5,1.5) circle (0.1cm);
\filldraw[gray] (3.5,3) circle (0.1cm);

\draw[thick, gray, dotted] (-1,0) -- (-0.5,0) -- (-0.5,-0.5);
\filldraw[gray] (-0.5,0) circle (0.1cm);
\end{tikzpicture}
\end{center}
\caption{\label{fig:quadrants}Body of a heart, with quadrants and extremal points. The heart is involved in conflicts in quadrants $a_1, b_2$ and $a_4$. The extremal point for $b_4$ is uninvolved because of rule (III), and extremal points for $a_2$ and $a_3$ are uninvolved because of rule (II). Extremal points for  $b_2$ and  $a_4$ satisfy (II) and (III).}
\end{figure}

\begin{lemma}
   \label{lem:bound_proba_color}
At any time $t$, the values of every central and side bit are independent of each other. Furthermore there exists a constant $0<\alpha\leq 1/2$ such that for any of these bits, the probability that it is worth $1$ at time $t$ is bounded between $\alpha$ and $1-\alpha$.
\end{lemma}

\begin{proof}
At time $0$, the central and side bits of each heart are determined by the initial Bernoulli measure, so the lemma is verified and the value of $\alpha$ is determined by the initial measure. A side bit is never changed and only interacts with the outside when its heart is destroyed, so we only prove the result for central bits. Since the only opportunity for central bits to be influenced by or to influence other bits is during conflict resolution, we prove the result by induction on the generation $n$ (with the basic case of the $0$-generation already proved).

Assume the result holds for some generation $n$ and consider a heart alive at time $t_{n+1}$. If the heart is not involved in a conflict at generation $n+1$, then its central bit is unchanged and doesn't influence any outside process, so the property is maintained. If it is involved in $k$ conflicts, then it either dies (in which case there is nothing to prove) or wins them all. 

Denote $\beta_1,\dots, \beta_k$ the side bits it received from his victims. The value of these side bits did not influence the conflict resolution and remains independent of every other bit since they are not sent to any other conflict winner (there is at most one conflict per quadrant). Furthermore:

\begin{align*}\mu\left(\sum_{i\geq 1}\beta_i = 1 \mod 2\right) &= \mu(\beta_1 = 1)\cdot\mu\left(\sum_{i\geq 2}\beta_i = 0 \mod 2\right) + \mu(\beta_1 = 0)\cdot\mu\left(\sum_{i\geq 2}\beta_i = 1 \mod 2\right)\\&\geq \alpha\cdot\mu\left(\sum_{i\geq 2}\beta_i = 0 \mod 2\right) + (1-\alpha)\cdot\mu\left(\sum_{i\geq 2}\beta_i = 1 \mod 2\right)\\&\geq \min(\alpha,1-\alpha)\geq\alpha \end{align*}
and symmetrically, using the induction hypothesis of independence between the values of all $\beta_i$.  \qed\end{proof}

Thanks to this process we can bound the \emph{radius} of an organism, which is the largest distance from a cell of its territory to its heart. 

\begin{definition}
An organism is \emph{healthy} if its radius is less than $K_n = 2^{n^{d-\frac 12}}$.
\end{definition}

\begin{lemma}
  \label{lem:density_hearts}
  Any cell belongs to a healthy organism with asymptotic probability $1$, i.e.:
  \[\max_{t_n\leq t\leq t_{n+1}}\mu\left(\left\{c\in\A^{\Z^d}:\exists x\in\Z^d, d_\infty(x,0)\leq K_n,\ \porga(F^t(c)_x)=\stp\right\}\right)\xrightarrow[n\to\infty]{} 1.\]
\end{lemma}

\begin{proof}
Consider a heart in $x\in\Z^d$ at time $t_n$. Its survival until time $t_{n+1}$ only depends on the hearts that are present at distance $2n$ or $2n+1$. Inductively, the survival of an initialised heart until the $n$-th generation depends only on hearts located at distance at most $\frac{(2n+1)(2n+2)}{2}$: the survival of distant enough hearts is independent.

Denote:\[\nu_n=\max_{t_n\leq t\leq t_{n+1}}\mu\left(\left\{c\in\A^{\Z^d}:\porga(F^t(c)_0)=\stp\right\}\right).\]

Take the hypercube of side length $2K_n+1$ centred on $0$ and cut it into $\lambda_n=\left(\left\lfloor\frac{2K_n+1}{(2n+1)(2n+2)+1}\right\lfloor\right)^d$ hypercubes of size $(2n+1)(2n+2)+1$. Consider the centres of these hypercubes $x_1,\dots,x_{\lambda_n}$. Then we have:
\begin{align*}1-\max_{t_n\leq t\leq t_{n+1}}\mu\left(\left\{c\in\A^{\Z^d}:\exists x\in\Z^d, d_\infty(x,0)\leq K_n ,\porga(F^t(c)_x)=\stp\right\}\right) &\leq (1-\nu_n)^{\lambda_n}\end{align*} where the second step uses the last remark and the shift-invariance of $\mu$.

For any $k\in\N$, a living heart at time $t_{k-1}$ survives until generation $k$ if it wins every conflict it is engaged in. Using Lemma~\ref{lem:bound_proba_color}, we have $\nu_k\geq \alpha^{3^d-1}\nu_{k-1}$ for some fixed $0<\alpha\leq 1/2$. Then by recurrence $\nu_n\geq \alpha^{n(3^d-1)}\mu(\{\sti\})$.

  We conclude that $(1-\nu_n)^{\lambda_n}\leq \left(1-\alpha^{n(3^d-1)}\mu(\{\sti\})\right)^{\lambda_n} \to 0$ since $\lambda_n\alpha^{n(3^d-1)}\mu(\{\sti\})\to\infty$.

 \qed\end{proof}


\subsection{Individual organisms: internal metabolism}
The last group of layers is used to govern the internal metabolism of the organisms. In this section, we consider some organism and describe how it behaves during a generation.

\subsubsection{Computing}\label{sec:Computing}
%
%

In this section, we describe the computational layer using the alphabet $\Acomp$. Let $(w_n)_n$ be the uniformly computable sequence of patterns given as an hypothesis of the theorem. Our goal is to delimit a small computation space around the heart where each $w_n$ will be computed in succession. 

We use standard techniques to embed the time evolution of any Turing machine $TM = (Q,\Gamma,\#,q_{0},\delta,Q_{F})$ inside our cellular automaton. The alphabet used for the simulation is $(\Gamma\cup\#)\times(Q\cup\#)$: the left part contains the tape symbol, and the right part contains the current state for the cell where the head is located, and $\#$ everywhere else. Then each step of the Turing machine moves the head and modifies the tape around the head according to local information, which can be done through the local rule of a CA.

\paragraph{A brain around each heart} The alphabet $\Acomp$ is divided into 4 sublayers containing various computational processes taking place in parallel next to the heart, including Turing machines with a $d$-dimensional tape. Each layer can read the contents of another sublayer when indicated. Assume that we are at time $t=t_n$.

\begin{enumerate}
\item The first layer contains a binary counter keeping track of the current generation;
\item The second layer contains a binary counter keeping track of the current value of $t$, similarly to binary age counters;
\item The machine on the third layer computes the value of $t_{n+1}$, then keeps watch on the time counter on the second layer. When it reaches $t=t_{n+1}$, the generation counter is incremented by one, which triggers many other processes.
\item The machine of the fourth layer reads the generation counter and computes the hypercubic patterns $w_n$ along with its side length $k$. $w_n$ is output on the main layer.
\end{enumerate}

From now on, we fix $t_{n} = \sum_{k<n}2^{k^{d-\frac 14}}$.

\paragraph{Complexity analysis} We want to ensure that these computations can be performed between times $t_n$ and $t_{n+1}$ without leaving a hypercube centred on the heart of side length $n^{\frac{d-\frac 14}d}$, that is, that they can be performed in time $t_{n+1}-t_n = 2^{n^{d-\frac 14}}$ and space $n^{d-\frac 14}$.

To get rid of the multiplicative constant contained in the $O$ notation, we use the standard techniques of linear speedup and tape compression for Turing machines. For any fixed constant $C$, by grouping cubes $C^d$ tapes cells together in a single letter and performing $C$ computation steps at once, we can divide required time and space by $C$. As downside, the tape alphabet of the Turing machines increases exponentially (in $C$). 
\begin{description}
\item[First and second layers:] This is obvious for the generation counter. The time counter occupies a space $\lceil\log t \rceil\leq \lceil\log t_{n+1} \rceil \sim \log(2^{n^{d-\frac 14}}) = O(n^{d-\frac 14})$, and the multiplicative constant is removed by using a base-$b$ counter with $b$ large enough.

\item[Third layer:] Computing the value of $t_{n+1} = \sum_{k<n+1}2^{k^{d-\frac 14}}$ takes space and time $O(2^{n^{d-\frac 14}})$.

\item[Fourth layer:] Without loss of generality, $w_n$ satisfies the time and space constraints, as the following Lemma shows.

\end{description}

We also assume that $w_n\in\A^{[0,k]^d}$ for some $\frac 12n^{\frac {d-\frac 12}d}< k\leq n^{\frac{d-\frac 12}d}$, by replacing $w_n$ by concatenating copies of itself if necessary.\bigskip

\begin{lemma}
  Given a computable sequence $(w_n)_n$ of hypercubes, there exists another computable sequence $(w'_n=w_{g(n)})_n$ such that:
  \begin{itemize}
  \item $g:\N\to\N$ is surjective and non-decreasing;
  \item $w'_n$ is computable in time $O(2^{n^{d-\frac 14}})$ and space $O(n^{d-\frac 14})$.
  \end{itemize}
  
\end{lemma}

\begin{proof}
  Consider a TM $\phi_0$ that computes the sequence $(w_n)_n$. We describe another machine $\phi$ on two tapes that computes $(w'_n)_n$. 
  
  We define the computation of $\phi$ on input $n$ inductively:
  \begin{itemize}
  \item compute $\phi(n-1)$ to obtain the value of $w'_{n-1}$ and $g(n-1)$ (they may be respectively empty and $0$, for example when $n=0$);
  \item draw an hypercube of side $s_n^{lim} = n^{\frac 1d(d-\frac 14)}$;
  \item on the second tape, compute and store the value $t_n^{lim} = n^{-d}2^{n^{d-\frac 14}}$ and initialise a counter with value $0$;
  \item simulate $\phi_0$ on input $g(n-1)+1$. Between each simulated step, increment the counter and compare it with $t_n^{lim}$;
  \item if the counter reaches the value $t_n^{lim}$ before $\phi_0$ halts, output $\phi(n-1)$;
  \item if on the other case $\phi_0$ halts first, output $w'_{n}$ (the output of $\phi_0$) and $g(n)=g(n-1)+1$.
  \end{itemize}

Hence $g$ is surjective and non-decreasing, and $(w'_n)_n$ is by construction computable in space $n^{d-\frac 14}$. Computing the value of $s_n^{lim}$ and $t_n^{lim}$ can be done in $O(2^{n^{d-\frac 14}})$ operations. Since $t_n^{lim}$ and the counter have length less than $n^d$, each incrementation and comparison step takes $O(n^d)$ steps at most. Therefore we have $O(n^dt_n^{lim}) = O(2^{n^{d-\frac 14}})$ operations outside of the recursive call, and the recursive call takes $O(\sum_{i}^{n-1}2^{i^{d-\frac 14}})=O(2^{n^{d-\frac 14}})$ operations.
\qed\end{proof}

To conclude, the described computations are doable within these time and space constraints, and $w_n$ is computed before time $t_{n+1}$. At time $t_{n+1}$ the second machine enters a special set of states that triggers various processes: organism-building signals, body-building, and the object of the next section, a copying process that will write concatenated copies of the pattern $w_n$ all over the main layer of the territory of the organism.

The alphabet $\Acomp$ is thus $\{0,1,2,\Bl\}\times \{0,1,2,\Bl\}\times(Q_3\cup\#)\times (\Gamma_3\cup\#)\times(Q_4\cup\#)\times (\Gamma_4\cup\#)$, where $Q_i,\Gamma_i$ are the state space and the tape alphabet of (the compressed version of) the $i$-th Turing machine described above.

\subsubsection{Copying}

The copy layer aims at copying the pattern $w_n$ output by the computational layer on the whole territory of the organism. In this section, auxiliary symbols belong to the copy layer $\Acopy$ but the pattern is written in the main layer with alphabet $\B$.

\paragraph{Writing grid} Remember that we assume $w_n\in\A^{[0,k]^d}$ for some $\frac 12 n^{\frac {d-\frac 12}d}< k\leq n^{\frac{d-\frac 12}d}$. Assume the central heart is located at $0$ for readability, and that the borders of $w_n$ have been marked with a special symbol $\stg$ (on the copy layer) by the Turing machine. The copying process relies on an (imaginary) cubic grid of side length $k$ that covers the whole territory of the organism. Starting from the cells centred on the heart, the pattern $w_n$ is copied in each cell of this grid passing from neighbour to neighbour, through a translation of vector $ke_j$ or $-ke_j$ for each $1\leq j\leq d$.

For some coordinates $i\in\Z^d$, define the corresponding grid element $\Sigma_i=\{\sum_{1\leq j\leq d}\alpha_je_j:\forall 1\leq j\leq d, ki_j\leq \alpha_j\leq ki_j+k\}$, and $\overline{\Sigma_i}$
its border (extremal cells). Notice that the computed pattern is supported by $\Sigma_{0,\dots, 0}$.

\paragraph{Local copy operation} For $u\in\mathcal{U}nit(d)$ and $i\in\Z^d$, we define the \emph{copy operation} $\mathcal{C}_i(u)$ that copies the contents of the main layer from $\Sigma_i$ to $\Sigma_{i+u}$. It consists in simulating a Turing machine (see previous section) that receives as input $k$ the side length of $w_n$ and works as follows:
 
 \begin{description}
 \item[Reproducing the borders:] The first step is to write $\stg$ in every cell of $\overline{\Sigma_{i+u}}$. It then travels to the coordinate $k(i+u)$ and builds an hypercube of symbols $\stg$ of side length $k$ (corresponding to $\overline{\Sigma_{i+u}}$). This takes $O(k^d)$ time steps. If the new hypercube is not entirely included in the territory of the organism, the copy process stops.
 \item[Reproducing the pattern:] The second step is to copy the pattern letter by letter. The machine copies each letter in lexicographic order, marking with a symbol letters already copied. Each letter needs at most $O(k)$ steps to be copied, so the whole process takes $O(k^{d+1})$ steps.
 \item[Cleaning the auxiliary states:] The third step is to remove all the auxiliary states that remain on the tape in the original grid hypercube $\Sigma_i$ (including $\stg$). This is done by going through all $k^d$ cells of $\Sigma_i$, taking $O(k^d)$ steps.
 \item[Selecting heirs:] The process spawns new copy processes from the hypercube $\Sigma_{i+u}$, transmitting along the value of $k$. The direction of the next processes are the following:
    \[\mathcal{C}_i(u)\rightarrow \left\{\begin{array}{ll}\{C_{i+u}(v) : v = \lambda_je_j + \sum_{k\neq j}\lambda_ke_k, \lambda_k\in\{-1,0,+1\}\}&\mbox{if }u=\lambda_je_j, \lambda_j\in\{-1,+1\} \\C_{i+u}(u)&\mbox{otherwise}\end{array}\right.\]
Those new processes are performed in parallel by duplicating the copy layer $2^d$ times.
 \end{description}
 
At the initial step, it is enough to trigger a copy process in all directions $u\in\mathcal{U}nit(d)$. The copying operations then progressively fill the whole organism, as can be seen in Figure~\ref{fig:copying}. 

 \begin{figure}[htbp] \centering
  \includegraphics[width=6cm]{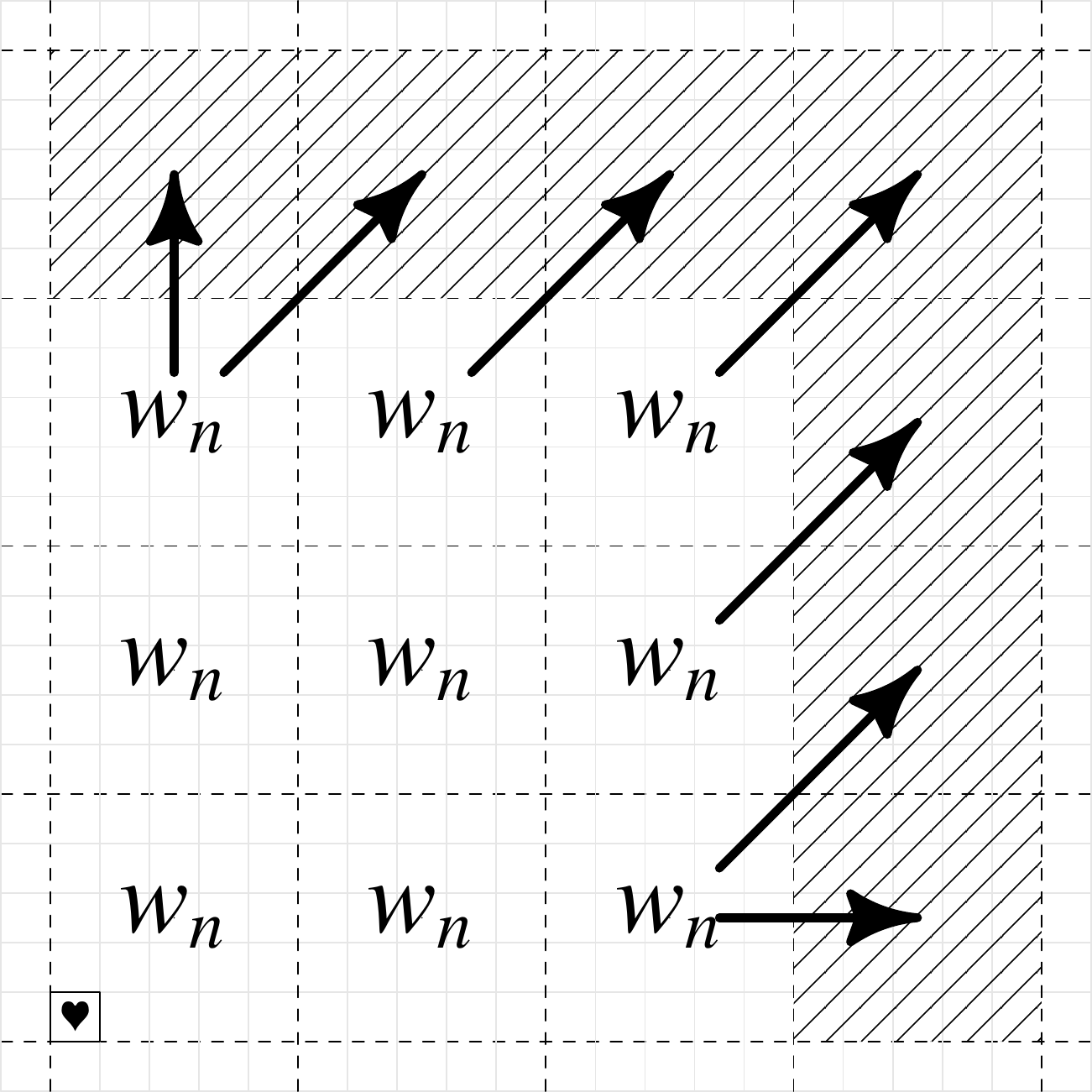}
  \caption{In this $2$-dimensional example, the pattern $w_n$ is copied from the heart of the organism towards its boundaries in successive steps.}
  \label{fig:copying}
\end{figure}

Each copying operation takes $O(k^{d+1})$ steps, and the active copying operations expand outward from the heart as a (thick) hypercube. Therefore, if the radius of the organism is $r$, the total time needed to finish the copying process is $\frac rk\cdot O(k^{d+1})$. We can take $r\leq 2^{n^{d-\frac 12}}$ by Lemma~\ref{lem:density_hearts} and $k\leq n^{\frac{d-\frac 12}d}$, which gives a total time of $O(n^d2^{n^{d-\frac 12}}) = O(2^{n^{d-\frac 14}})$. Lowering if needed the multiplicative constant by the linear speedup theorem, we see that the process ends before time $t_{n+1}$.

\subsection{Proof of the main theorem}

We first prove that the density of auxiliary states tend to 0 as time tends to infinity, which ensures they are not charged by any limit measure.

\begin{lemma}
  \label{lem:density_bord}
  Border symbols $\std$ have negligible density asymptotically, i.e., for any nondegenerate Bernoulli measure $\mu$:
  \[F^t\mu\left(\left\{c\in\A^{\Z^d}: \porga(c)_0=\std\right\}\right)\to_t 0.\]
\end{lemma}

\begin{proof}
  By Lemma~\ref{lem:life_supremacy}, we can consider only the cells in the colonised space, i.e. inside a living membrane.
  Given a configuration $c$ and some time $t_n\leq t<t_{n+1}$ during the $n$-th generation, denote $S_{\stdsub}(t)=\{x\in\Z^d: \porga(F^t(c)_x)=\std\}\cap \Col_t(c)$ the set of colonised cells containing a border, and $\overline{S}_{\stdsub}(t)$ the complement of the previous set. We show that there exists a constant $\lambda$ such that $\frac{|S_{\stdsub}(t)|}{|\overline{S}_{\stdsub}(t)|}\leq \frac{\lambda}{n}$. This property being true for every initial configuration $c$, the lemma follows using Birkhoff's theorem.

To do this, we show in the next Lemma that borders between organisms can be partitioned into (subsets of) hyperplanes. To each such hyperplane subset bordering two organisms, we associate some volume inside the territory of one of the organisms that is $\frac n\lambda$ times larger than the subset itself. 

\begin{lemma}
  The common border of two hearts is included in a (finite) union of hyperplanes of $\mathcal{H}yp(d)$. (Recall $\mathcal{H}yp(d)$ is the set of hyperplanes with a normal vector in $\mathcal{U}nit(d)$.)
\end{lemma}

\begin{proof}
  Let $x$ and $x'$ be two cells containing a heart each. If no other heart existed in the whole space, any point $y$ in the border between these two hearts would verify $|d_\infty(x,y) - d_\infty(x',y)|\leq 1$. Taking $i$ and $j$ such that $d_\infty(x,y)=|y_i-x_i|$ and $d_\infty(x',y)=|y_j-x'_j|$, this border would be included in $\displaystyle \bigcup_{\substack{1\leq i,j\leq d\\\varepsilon=\pm 1}}H_{i,j}^\varepsilon$, where $H_{i,j}^\varepsilon$ are defined as:
  
\[y\in H_{i,j}^\varepsilon \Longleftrightarrow \left\{
\begin{array}{ll}
|2y_j-x_j-x'_j|\leq 1&\quad\mbox{if }i=j\mbox{ and }x_j = x'_j\\
|y_i-x_i - \varepsilon(y_j-x'_j)|\leq 1 &\quad \mbox{otherwise} \end{array}
\right.\]
All these sets are unions of one or two hyperplanes of $\mathcal{H}yp(d)$ (depending on the parity of $x_i - \varepsilon x'_j$). In the presence of other hearts, the border between $x$ and $x'$ is a subset of this "ideal border", which proves the Lemma.
\qed\end{proof}

Given two hearts at cells $x_0$ and $x_1$, denote $B(x_0,x_1)$ the set of cells corresponding to their common border. Partition this set into a finite collection $\{H_1,\dots, H_k\}$ of disjoint subset of hyperplanes according to the previous lemma. For each such $H_i$, as $d_\infty(x_0,x_1)\geq 2n$, we have either $d_\infty(x_0,H_i)\geq n$ or  $d_\infty(x_1,H_i)\geq n$.

 Take any $i$ and any finite subset $s$ of $H_i$. Denote $A(s)$ the area of $s$ and $V_0$ and $V_1$ the volumes of the $d$-polytopes limited by the surface $s$ and the points $x_0$ and $x_1$, respectively. Then $V_{\iota}=\frac{1}{d}d_\infty(x_{\iota},H_i)A(s)$ for each $\iota\in\{0,1\}$. Denote $V(s)=V_0+V_1$, then $\frac{A(s)}{V(s)}\leq \frac{d}{n}$.

  We now do this operation for every organism, that is split $S_{\stdsub}(t)$ into a collection $\mathcal{S}$ of disjoint hyperplanar surfaces that belong to the common border of two organisms. For every two such different surfaces, the corresponding volumes inside organisms are also disjoint, hence
  {\begin{align}
    \frac{|S_{\stdsub}(t)|}{|\overline{S}_{\stdsub}(t)|}&\leq \frac{\sum_{\mathcal{S}}A(s)}{\sum_{\mathcal{S}}V(s)}\\
    &\leq \frac{d}{n}
  \end{align}}\qed
\end{proof}

\begin{lemma}
\label{lem:no_auxiliary_states}
For any nondegenerate Bernoulli measure $\mu$ and $z\in\Z^d,$
\[\mu\left(F^t(c)_z\in\A\setminus\B\right)\underset{t\to\infty}\longrightarrow 0.\]
\end{lemma}
\begin{proof} We handle each layer separately.
\paragraph{Uncolonised space and membranes} First, by Lemma~\ref{lem:life_supremacy}, we can see that $c_0$ belongs to the uncolonised space at time $t$ only if the nearest viable seed at time 0 is at distance more than $\sqrt{t}$. For the same reason, $c_0$ can be part of a living membrane or a related process (age counter, respiration process, comparison process) only if the nearest viable seed is at distance more than $\sqrt{t} - \log t$. Since a viable seed appear with a nonzero probability, the probability of this event tends to 0 as t tends to infinity.\bigskip{}

It remains to handle symbols appearing inside the colonised space on the layers dealing with internal affairs of the colonies: organism, evolution, computing, copying and main layers. By Birkhoff's ergodic theorem, it is equivalent to prove that the density of auxiliary states in a configuration tends to 0 almost surely when time tends to infinity.

\paragraph{Hearts, computing symbols} In the colonised space hearts $\stp$ must be issued from a seed, and as explained in Section~\ref{sub:colonies} they each have at time $t_n$ a body, which are non-overlapping hypercubes of side $2n+1$ centred on the heart (more precisely, they can overlap shortly but are destroyed before the next $t_n$). Thus the density of hearts $\stp$ in $c$ between times $t_n$ and $t_{n+1}$ is less than $\frac 1{(2n-1)^d}$. Since the computing process taking place around the heart is contained in a hypercube of side $\sqrt n$, the density of cells with nonempty computing layer is almost surely less than $\frac 1{(2n-1)^{d/2}}$ in this period. 

\paragraph{Bodies and bodybuilding signals} Body symbols form the surface of an hypercube of side $2n+1$ (when it is fully built) or less (during the construction), and therefore there are less than $2d(2n+1)^{d-1}$ such symbols for each heart. The impulses used to grow the body being sent one at a time, they occupy at most as much space as the body itself at any given time. Therefore all those symbols have density less than $2d(2n+1)^{d-1}\cdot \frac 1{(2n-1)^d} = O\left(\frac 1n\right)$.

\paragraph{Borders and border-building signals} Borders $\std$ were handled in Lemma~\ref{lem:density_bord}. We use a similar argument to show that the density of symbols in signals used to build borders is asymptotically negligible. The signal is born around the heart and progresses at speed one. Therefore, $m$ steps after its birth, the set of cells in the organism containing the signal is an hypercube of side $2m+1$ centred on the heart (intersected with the inside of the organism). In particular, in an organism of healthy size, the signal sent at time $t_n$ has disappeared before time $t_n+2^{n^d}\leq t_{n+1}$, so at most one signal appears in a given organism at the same time.

If $m\leq n$, since the organism contains at least $n^d$ cells, signal symbols have density less than $\frac {2d(2m+1)^{d-1}}{n^d} = O(\frac 1n)$. 
If $m>n$, notice that for each cell $z$ of the organism satisfying $d_\infty(\stp,z)=m$, the line between $\stp$ and $z$ is contained in the organism (by Lemma \ref{lem:triangul}) and does not contain other signal symbols (since its distance to the heart is less than $m$). For any part $P$ of the surface area of the hypercube which is inside the organism, the convex hull of $P$ and $\stp$ is inside the organism as well and its interior does not contain any symbol. The proportion of $P$ to the total surface area is the same as the proportion of its convex hull to the total volume. Therefore the symbol density is at most $\frac {2d(2m+1)^{d+1}}{m^d} = O\left(\frac 1n\right)$ (since $m>n$).

\paragraph{Copying processes}The copying grid $\stg$ is simply a grid of side length $\sqrt{n}$, and therefore the density of symbols $\stg$ is less than $\frac {2d(2n+1)^{d-1}}{n^d} = O\left(\frac 1n\right)$. Each copying operation contains symbols in at most two squares at any given time: one from which it copies and one to which it copies. Furthermore, because all copying operations take the same amount of time $C(n)$ to copy one square, the whole copying process of an organism in the time interval $[t_n+kC(n), t_n+(k+1)C(n)]$ is contained in the squares located at "distance" $k$ and $k+1$ from the heart, i.e. the cells whose distance from the heart is between $k\sqrt n$ and $(k+2) \sqrt n$. Furthermore, by Lemma~\ref{lem:density_hearts}, the probability that an organism contains only one copying process tends to 1. The previous argument (used for the border-building signals) shows that copying symbols have density at most $O\left(\frac{2\sqrt n}n\right) = O\left(\frac 1{\sqrt{n}}\right)$.
\qed\end{proof}

From this lemma, we see that no limit measure can assign a non-zero probability to any pattern with a non-empty auxiliary layer.

\begin{lemma}
\label{lem:close_to_polygonal_path}
\[\dm(F^{t_n}\mu, \meas{w_n})\underset{n\to\infty}\longrightarrow 0\hspace{1cm} \mbox{and}\hspace{1cm} \max_{t_n\leq t\leq t_{n+1}} \dm\left(F^t\mu, [\meas{w_n}, \meas{w_{n+1}}]\right) \underset{n\to\infty}\longrightarrow 0.\] 
\end{lemma}

\begin{proof}
Take any finite square pattern $u\in\A^{[0,\ell]^d}$. From Lemma~\ref{lem:no_auxiliary_states} and by $\s$-invariance, we can see that if $c$ is drawn according to $\mu$ then the probability that $F^t(c)_{[0,\ell]^d}$ has any part outside the colonised space or with a nonempty auxiliary layer is $O\left(\frac 1{\sqrt n}\right)$. Inside any organism at time $t=t_n$, the main layer contains concatenated copies of $w_{n-1}$ in all directions except for those cells at distance more than $\frac{t_n-t_{n-1}}{C(n)}\sqrt{n}$ from the heart (see the last paragraph of the previous proof), which forms an asymptotically negligible set by Lemma~\ref{lem:density_hearts}. By $\s$-invariance, we obtain that:
\[|F^{t_n}\mu([u]) - \meas{w_n}([u])|\underset{n\to\infty}\longrightarrow 0.\]
Since this is true for any square pattern, we get the first part of the result.\bigskip{}

At time $t_n$, the copying process for $w_n$ is triggered. As explained in the last paragraph of the previous proof, between times $t_n+kC(n)$ and $t_n+(k+1)C(n)$ the copying process is contained in cells at distance $k\sqrt{n}$ to $(k+2)\sqrt{n}$ from the nearest heart. In particular, the main layers of cells at distance less than $k\sqrt{n}$ from the nearest heart contain concatenated copies of $w_n$ while those at distance more than $(k+2)\sqrt{n}$ still contain concatenated copies of $w_{n-1}$. 

Therefore, denoting by $h(c)$ the minimum distance between 0 and an heart in $c$, we have for any $t_n\leq t\leq t_{n+1}$: 
\[F^{t}\mu([u]) = \mu\left(h(c)\leq \frac {t-t_n}{C(n)}\sqrt{n}\right)\cdot\meas{w_n}([u]) + \mu\left(h(c)> \frac {t-t_n}{C(n)}\sqrt{n}\right)\cdot\meas{w_{n-1}}([u])+\underset{n\to\infty}o(1).\]
The second term contains $\mu\left(h(c)> \frac {t-t_n}{C(n)}\sqrt{n}\right)$ instead of the expected $\mu\left(h(c)> \left(\frac {t-t_n}{C(n)}+2\right)\sqrt{n}\right)$ to get an actual barycentre, the difference between them is asymptotically negligible in $n$. This equation holding for any square pattern $u$, we obtain:
\begin{equation}\dm\left (F^t\mu\ ,\ \mu\left(h(c)\leq \frac {t-t_n}{C(n)}\sqrt{n}\right)\cdot\meas{w_n} + \mu\left(h(c)> \frac {t-t_n}{C(n)}\sqrt{n}\right)\cdot\meas{w_{n-1}}\right)\underset{n\to\infty}\longrightarrow 0. \label{eq:distance}\end{equation}
The right-hand measure belonging to the segment $[\meas{w_{n-1}}, \meas{w_n}]$, and this being true for any $t_n\leq t\leq t_{n+1}$, we obtain the desired result.
\qed\end{proof}
\begin{proof}[of Theorem~\ref{thm:MainTheorem}]
By the right-hand part of Lemma~\ref{lem:close_to_polygonal_path}, we see that $\V(F,\mu)$ is included in the closure of the polygonal path delineated by the sequence $(\meas{w_n})_{n\in\N}$. We prove the other inclusion.

Take any $\nu\in [\meas{w_{n-1}}, \meas{w_n}]$. For $t_{n-1}\leq t\leq t_n$, denote $\mu_t$ the closest point to $F^t\mu$ in $[\meas{w_{n-1}}, \meas{w_n}]$; by Lemma~\ref{lem:close_to_polygonal_path}, $\dm(F^t\mu, \mu_t) \to 0$. We prove that $\nu$ is close to one of the $\mu_t$. By Equation~(\ref{eq:distance}), we have $\dm\left(F^t\mu, F^{t+1}\mu\right)\leq 2\mu\left(\frac {t-t_n}{C(n)}\leq h(c) \leq\frac {t+1-t_n}{C(n)}\right) +\underset{n\to\infty}o(1) \to 0$. Since $\dm\left(F^t\mu,\mu_t\right)=\underset{n\to\infty}o(1)$, it follows that $\dm(\mu_t, \mu_{t+1}) \underset{n\to\infty}o(1)$ as well.

Since $\dm\left(\mu_{t_i},\meas{w_i}\right)=\underset{n\to\infty}o(1)$ for any $i$ and $\dm(\mu_t, \mu_{t+1})\to 0$, it follows that \[\min_{t_n\leq t\leq t_{n+1}}\dm(\nu, \mu_t) = \underset{n\to\infty}o(1)\qquad \text{ and thus }\qquad\min_{t_n\leq t\leq t_{n+1}}\dm(\nu, F^t\mu) = \underset{n\to\infty}o(1).\]

This proves the other inclusion.
\qed\end{proof}

\section{Statement of the results}
\label{sec:results}
From Theorem~\ref{thm:MainTheorem} we deduce a number of results which are our main contributions.

\begin{corollary}
The measures $\nu\in\Ms(\azd)$ for which there exist:
\begin{itemize}
\item an alphabet $\B\supset \A$,
\item a cellular automaton $F:\bzd\to\bzd$, and
\item a non-degenerate Bernoulli measure $\mu\in\Ms(\bzd)$
\end{itemize}
such that $F^t\mu\xrightarrow[t\to\infty]{}\nu$, are exactly the limit-computable measures.
\end{corollary}

\begin{corollary}
The connected sets of measures $\K\subset\Ms(\azd)$ for which there exist:
\begin{itemize}
\item an alphabet $\B\supset \A$,
\item a cellular automaton $F:\bzd\to\bzd$, and
\item a non-degenerate Bernoulli measure $\mu\in\Ms(\bzd)$
\end{itemize}
such that $\V(F,\mu)=\K$, are exactly the $\Pi_2$-computable, connected, compact sets of measures.
\end{corollary}

Furthermore, both corollaries hold if one requires the convergence to hold for all nondegenerate Bernoulli measures.

\begin{proof}
Apply Theorem~\ref{thm:MainTheorem} to Proposition~\ref{prop:polygonalcover}. To get Corollary 1, use the fact that $\nu$ is a limit-computable measure if and only if the singleton $\{\nu\}$ is a $\Pi_2$-computable set of measures (and of course connected).
\qed\end{proof}

Following \cite{HellouinSablik}, we obtain a similar characterisation using convergence in Cesàro mean (Corollary 5 in op.cit.) and a Rice-style theorem on $\mu$-limit measures set (Corollary 7 in op.cit.). Since the proofs of op.cit. only involve finding an appropriate uniformly computable sequence $(w_n)$ without modifying the cellular automaton, they can be carried straightforwardly to the $d$-dimensional case by replacing $\A^\Z$ by $\azd$ and we do not repeat them here.

\begin{corollary}
The sets of measures $\K'\subset \K \subset\Ms(\azd)$ for which there exist:
\begin{itemize}
\item an alphabet $\B\supset \A$,
\item a cellular automaton $F:\bzd\to\bzd$, and
\item a nondegenerate Bernoulli measure $\mu\in\Ms(\bzd)$
\end{itemize}
such that $\V(F,\mu)=\K$ and $\V'(F,\mu) = \K'$, are exactly the $\Pi_2$-computable, connected, compact sets of measures.
\end{corollary}

In particular we characterise all sets of measures reachable at the limit in convergence in Cesàro mean from a Bernoulli measure, since those sets are necessarily connected (Section 1.2.3 in op.cit.). Here again, the result holds if one requires the convergence to hold for all nondegenerate Bernoulli measures.

\begin{corollary}
Let $P$ be a nontrivial property (i.e. not always or never true) on non-empty $\Pi_2$-computable, compact, connected sets of $\Ms(\azd)$. There is no algorithm that can decide, given an alphabet $\B$, a cellular automaton $F:\bzd\to\bzd$ and a Bernoulli measure $\mu\in\Ms(\bzd)$, whether $\V(F,\mu)$ satisfies $P$.
\end{corollary}

Here it is assumed that the Bernoulli measure is finitely described by a list of (rational) parameters. A similar statement follows on nontrivial properties of limit-computable measures. This corollary would also hold if the property was required to hold, not only for one, but for some or all nondegenerate Bernoulli measure(s).

\section{Open questions}
The main questions that remain open concern the characterisation of non-connected sets of limit measures and the extension to more general sets of initial measures.  In particular, the result in the $1$-dimensional case holds for a large diversity of initial measures ($\sigma$-mixing with full support), which we could not obtain for the lack a finer analysis of the disappearance rate of the hearts.

A longer-term research direction concerns surjective cellular automata. The construction developed here is intrinsically non-surjective, and it does not seem that it can be adapted easily (in particular due to the key role of a time $0$). Surjective cellular automata are known to be Turing-universal in the classical sense, but surjectivity has a deep impact on the dynamics of the model which is specific to the probabilistic setting \cite{Kari2015}. In some sense, the question is whether this dynamical restriction is strong enough to lower the computing power of the model. Even seemingly simple questions, such that the existence of a non fully-supported set of limit measures, remain open.

\bibliographystyle{spmpsci}
\bibliography{biblio}

\end{document}